\documentclass[11pt,letterpaper]{article}
\usepackage[english]{babel}
\usepackage{amsmath,amsthm}
\usepackage{amsfonts, mathrsfs}
\usepackage{amssymb}
\usepackage{verbatim}
\usepackage{graphicx}
\usepackage{color}
\usepackage{tikz-cd}
\usepackage{enumitem}
\usepackage{microtype}
\usepackage{url}

\allowdisplaybreaks

\numberwithin{equation}{section}
\newtheorem{theorem}{Theorem}[section]
\newtheorem{cor}[theorem]{Corollary}

\newtheorem{lemma}[theorem]{Lemma}

\theoremstyle{definition}
\newtheorem{defn}[theorem]{Definition}
\newtheorem{remark}[theorem]{Remark}
\newtheorem{examples}[theorem]{Examples}

\newtheorem*{assum*}{\assumptionnumber}
\providecommand{\assumptionnumber}{}
\makeatletter
\newenvironment{assum}[1]
 {%
  \renewcommand{\assumptionnumber}{Assumption #1}%
  \begin{assum*}%
  \protected@edef\@currentlabel{#1}%
 }
 {%
  \end{assum*}
 }
\makeatother

\renewcommand{\labelenumi}{\alph{enumi}.)}

\begin{document}

\title{Stochastic recursions on directed random graphs}

\author{Nicolas Fraiman \and Tzu-Chi Lin \and Mariana Olvera-Cravioto}

\maketitle

\begin{abstract}
For a vertex-weighted directed graph $G(V_n, E_n;\mathscr{A}_n)$ on the vertices $V_n = \{1,2, \dots, n\}$, we study the distribution of a Markov chain $\{ {\bf R}^{(k)}: k \geq 0\}$ on $\mathbb{R}^n$ such that the $i$th component of ${\bf R}^{(k)}$, denoted $R_i^{(k)}$, corresponds to the value of the process on vertex $i$ at time $k$. We focus on processes $\{ {\bf R}^{(k)}: k \geq 0\}$ where the value of $R_i^{(k+1)}$ depends only on the values $\{ R_j^{(k)}: j \to i\}$ of its inbound neighbors, and possibly on vertex attributes. We then show that, provided $G(V_n, E_n; \mathscr{A}_n)$ converges in the local weak sense to a marked Galton-Watson process, the dynamics of the process for a uniformly chosen vertex in $V_n$ can be coupled, for any fixed $k$, to a process $\{ \mathcal{R}_\emptyset^{(r)}: 0 \leq r \leq k\}$ constructed on the limiting marked Galton-Watson tree. Moreover, we derive sufficient conditions under which $\mathcal{R}^{(k)}_\emptyset$ converges, as $k \to \infty$, to a random variable $\mathcal{R}^*$ that can be  characterized in terms of the attracting endogenous solution to a branching distributional fixed-point equation. Our framework can also be applied to processes $\{ {\bf R}^{(k)}: k \geq 0\}$ whose only source of randomness comes from the realization of the graph $G(V_n, E_n;\mathscr{A}_n)$. 
\vspace{5mm}



\end{abstract}

\section{Introduction}

The main motivation for this work is to provide a mathematical framework that can be used to establish a rigorous connection between the stationary distribution of Markov chains on $\mathbb{R}^n$, whose dynamics are governed by the neighbor-to-neighbor interactions on a directed graph, and the solutions to branching distributional fixed-point equations. To illustrate the type of connection we seek, consider the Markov chain $\{ W_k: k \geq 0\}$ corresponding to the waiting time of the $k$th customer in a single-server queue with i.i.d.~interarrival times and i.i.d.~processing times. It is well-known that, $\{W_k : k \geq 0\}$ satisfies the recursion:
$$W_{k+1} = ( W_k + \chi_k - \tau_{k+1} )^+, \qquad k \geq 0,$$
where $\chi_k$ is the processing time of the $k$th customer and $\tau_k$ is the interarrival time between the $(k-1)$th and $k$th customers. We also know that provided that $E[ \chi_1 - \tau_1] < 0$, there exists a random variable $W_\infty$ such that $W_k \Rightarrow W_\infty$ as $k \to \infty$, where $\Rightarrow$ denotes weak convergence. It follows that $W_\infty$ can be characterized as the unique solution to Lindley's equation:
$$W \stackrel{\mathcal{D}}{=} ( W + \chi - \tau )^+,$$
where $W$ is independent of $(\chi,\tau)$ and $\stackrel{\mathcal{D}}{=}$ denotes equality in distribution.

In general, if $\{ X_k: k \geq 0\}$ is a Feller chain on $\mathbb{R}$ whose transitions are determined via
$$X_{k+1} = \Phi( X_k, \xi_k), \qquad k \geq 0,$$
for some deterministic and continuous map $\Phi$, and $\{ \xi_k: k \geq 0\}$ a sequence of i.i.d.~random variables, then, assuming $X_k \Rightarrow X_\infty$ as $k \to \infty$ for some random variable $X_\infty$, we can expect that $X_\infty$ will be a solution to 
$$X \stackrel{\mathcal{D}}{=} \Phi\left( X, \xi \right),$$
where $X$ is independent of $\xi$. 

Branching distributional equations take the general form
$$X \stackrel{\mathcal{D}}{=} \Psi \left( N, Q , \{ X_i, C_i: 1 \leq i \leq N \} \right),$$
where the $\{ X_i\}$ are i.i.d.~copies of $X$, independent of $(N, Q, C_1, C_2, \dots)$. These recursions often arise  when analyzing structures on trees \cite{Aldo_Band_05}, divide-and-conquer algorithms \cite{Fill_Jan_01, Devroye_01}, queueing networks with synchronization requirements \cite{Kar_Kel_Suh_94,Olv_Rui_20}, and as heuristics for the stationary behavior of random processes on sparse, locally tree-like graphs.

The goal of this paper is to prove that the stationary distributions of certain types of Markov chains defined on large locally tree-like directed graphs approximately satisfy branching distributional fixed-point equations.  The main practical implication of our result lies on the mathematical tractability of distributional fixed-point equations, which can then be used to analyze the properties of the stationary distributions of interest.

The set of techniques developed in this paper rely heavily on the analysis of the PageRank algorithm {\cite{Volk_Litv_10, Jel_Olv_10, Lee_Olv_20, Grav_etal_20, Olvera_20}}, which in the context of this work corresponds to a linear recursion with no randomness other than the one used to generate the graph where it is defined. The PageRank recursion has now been extensively studied, both from the point of view of its convergence to a branching distributional fixed-point equation, and also from the point of view of the tail behavior of its solutions {\cite{Alsm_Mein_10a, Alsm_Mein_10b, Olvera_12}}. Interestingly, much of what is known for PageRank can be extended to include additional random noises that act on either the vertices or the edges of the graph, and the mode of convergence can be extended, for unbounded recursions, to include the convergence of certain moments. When we add the noises, our framework connects to the study of discrete-time interacting particle systems like those studied in \cite{lacker2020local}, specialized to recursions that will lead to a characterization of their stationary distribution in terms of a branching distributional fixed-point equation. Specifically, the type of discrete-time interacting particle systems that our framework covers are defined on directed graphs with, asymptotically, no self-loops.

The remainder of the paper is organized as follows. In Section~\ref{S.Structure} we give a precise description of the type of recursions that we study, including some well-known examples that fit into our framework. In Section~\ref{S.Characterization} we specify the type of directed random graphs for which our results will hold and state our main theorem. In Section~\ref{S.EndogenousSol} we explain how to construct solutions to branching distributional fixed-point equations, and in Section~\ref{S.Proofs} we prove all our theorems. Finally, for completeness, we give in the {appendix} a brief description of two important families of random graph models known to satisfy all of our assumptions.

\subsection{Structure of the recursions} \label{S.Structure}

Consider a vertex-weighted directed graph $ G(V_n, E_n; \mathscr{A}_n)$ on the set of vertices $V_n = \{1, 2, \dots, n\}$, having directed edges on the set $E_n$, and where each vertex $i \in V_n$ has an attribute $\mathbf{a}_i \in \mathcal{S}'$, where $\mathcal{S}'$ is a separable metric space;  define $\mathscr{A}_n = (\mathbf{a}_i)_{i \in V_n}$. We can think of the set of attributes $\mathscr{A}_n$ as latent variables that will allow us to construct the graph $G(V_n, E_n; \mathscr{A}_n)$, e.g. in an inhohomogeneous random graph \cite{Hofstad1} these would include the weights that modulate the degrees, and in a configuration model \cite{bollobas, Hofstad1} these would correspond to the target degree sequence. In some of the applications that we have in mind, the attributes can also include parameters that are needed in the model, e.g., in personalized PageRank we include the personalization probability of each vertex as well as the damping factor. 

We write $j \to i$ to denote that there is a directed edge from vertex $j$ to vertex $i$. We will assume later that $\mathbf{G} := \{ G(V_n, E_n; \mathscr{A}_n): n \geq 1\}$ is a graph sequence converging in the local weak sense to a {marked Galton-Watson tree}, and we will analyze it under the large graph limit $n \to \infty$.  However, the Markov chain we analyze, which we describe in Section~\ref{SS.FixedGraph} below, is defined on a fixed graph.

\subsubsection{Markov chain on a fixed graph} \label{SS.FixedGraph}

Let $G(V,E;\mathscr{A})$ be a fixed vertex-weighted directed graph.  The graph does not need to be simple, nor is it required to be strongly connected. To each vertex $i \in V$ assign its full vertex mark, which includes its vertex attribute as well as its in-degree $d_i^-$ and its out-degree $d_i^+$:
$${\bf x}_i = (d_i^-, d_i^+, {\bf a}_i), \qquad i \in V.$$
Note that the full vertex marks take values on the space $\mathcal{S} := \mathbb{N} \times \mathbb{N} \times \mathcal{S}'$. Throughout the paper we use lower-case letters to refer to deterministic quantities and upper case ones to denote random elements.

To define the Markov chain $\{{\bf R}^{(k)}: k \geq 0\}$, start with an initial vector $\mathbf{R}^{(0)}$ consisting of i.i.d.~components $\{R_i^{(0)}: 1 \leq i \leq n\}$ distributed according to some initial distribution $\mu_0$ on $\mathbb{R}$, and recursively define 
$$R^{(k+1)}_i = \Phi\left( {\bf x}_i, \zeta_i^{(k)}, \left\{g(R_j^{(k)}, {\bf x}_j), \xi^{(k)}_{j,i}: j \to i\right\} \right), \qquad i \in V, \quad k \geq 0,$$
where the sequences $\{ \zeta_i^{(k)}: i \geq 1, \, k \geq 0\}$ and $\{\xi_{i,j}^{(k)}: i,j \geq 1, \, k \geq 0\}$ consist each of i.i.d.~random variables, independent of each other and of any other random variables in the graph and $g$ is a deterministic function. We will refer to the random variables $\boldsymbol{\zeta}^{(k)} = \{ \zeta_i^{(k)}: i \in V \}$ and the vectors $\boldsymbol{\xi}^{(k)} = \{ \xi_{i,j}^{(k)}: i,j \in V\}$ as the ``noise" at time $k$, with $\boldsymbol{\zeta}^{(k)}$ representing noise on the vertices and $\boldsymbol{\xi}^{(k)}$ noise on the edges. For aesthetic reasons we do not capitalize the noise sequences.

The description of $\{{\bf R}^{(k)}: k \geq 0\}$ is convenient from a modeling perspective since it makes explicit the role that the vertex marks $\{ \mathbf{x}_i \}$ play. However, it is equivalent to the two-step representation using the composition map $\Psi$ and the Markov chain $\{ \mathbf{V}^{(k)}: k \geq 0\}$ given by:
\begin{align*}
V_i^{(0)} &= g(R_i^{(0)}, \mathbf{x}_i), \\
V_i^{(k+1)} &= \Psi\left( \mathbf{x}_i,  \zeta_i^{(k)}, \left\{ V_j^{(k)} , \xi^{(k)}_{j,i}: j \to i\right\} \right) \\
&:= g\left( \Phi\left( \mathbf{x}_i, \zeta_i^{(k)}, \left\{ V_j^{(k)} , \xi^{(k)}_{j,i}: j \to i\right\} \right) , \mathbf{x}_i \right)  , \qquad k \geq 0,
\end{align*}
for each vertex $i \in V$. We can then compute $\mathbf{R}^{(k)}$ for $k \geq 1$ via
$$R_i^{(k+1)} = \Phi\left( \mathbf{x}_i, \zeta_i^{(k)}, \left\{ V_j^{(k)} , \xi^{(k)}_{j,i}: j \to i\right\} \right), \qquad i \in V.$$
It is in fact the map $\Psi$ that leads to a distributional fixed-point equation when $g$ is other than $g(r,\mathbf{x}) = r$.

Our framework also includes recursions on graphs that do not have any noises, in which case the iterative process, assuming it converges, results in a solution to the (possibly non-linear) system of equations:
$$r_i = \Phi\left( \mathbf{x}_i, \left\{g(r_j, \mathbf{x}_j): j \to i\right\} \right), \qquad i \in V. $$
In this case, the randomness leading to a distributional fixed-point equation will come only from the randomness in the graph sequence $\boldsymbol{G} = \{  G(V_n, E_n; \mathscr{A}_n) : n \geq 1\}$ that will be introduced later.

\begin{remark}
It is important to point out that the type of recursions we study exclude cases where $R_i^{(k+1)}$ depends on $R_i^{(k)}$, since the random graphs $G(V_n, E_n; \mathscr{A}_n)$  will not in general have self-loops. In other words, the evolution of each vertex is determined only by that of its inbound neighbors, their vertex attributes, its own vertex attribute, and potentially, noises on vertices and/or edges, but not by its own history.  It is this feature that allows the limit as $k \to \infty$ to be fully characterized in terms of a branching distributional fixed-point equation, which is not true for more general recursions.
\end{remark}

\begin{examples} \label{Ex:Models}
The following examples illustrate the type of recursions that our framework is intended to cover. The first two examples include noises at the vertices, while the second two have no noises of either kind and are described in terms of the system of equations they satisfy as $k \to \infty$. 
\begin{enumerate}[leftmargin=5mm] \renewcommand{\labelenumi}{\arabic{enumi}.} 
\item \textit{Generalized DeGroot model:} This model is used in social sciences to study opinion dynamics. Here, $R_i^{(k)}$ denotes the opinion of the $i$th individual in the population at time $k$,
    \begin{align}\label{eq:ExtDeGroot}
        R_i^{(k+1)}  & = c f\left(q_i,\zeta_i^{(k)} \right) + \frac{1-c}{d_i^-} \sum_{j \rightarrow i} R_j^{(k)},
    \end{align}
 where $c \in [0,1]$ is a damping factor and $q_i$ is a vertex parameter. When $c = 0$, equation \eqref{eq:ExtDeGroot} is known as the DeGroot model~\cite{degroot1974reaching}. When $f(q_i, \zeta_i^{(k)}) = q_i$, it is known as the Friedkin-Johnsen model \cite{friedkin1990social, friedkin1999influence, frasca2013gossips, ravazzi2014ergodic}, and $q_i$ represents the internal opinion or stubbornness of agent $i$. Another extension of the model \cite{yang2017innovation} sets $f(q_i, \zeta_i^{(k)}) = \zeta_i^{(k)}$ in order to model miscommunication between the agents. This choice of $f(q_i, \zeta_i^{(k)}) = \zeta_i^{(k)}$ has also been used in  \cite{acemoglu2015networks, damonte2019systemic} to model production networks, in which case $R_i^{(k)}$ represents the logarithm of the output of firm $i$, $\zeta_i^{(k)}$ models a shock to firm $i$, and $c$ is the level of interconnection in the economy. 
\item \textit{Noisy majority voter model:} This is an interacting particle system, with agents on a directed graph, and taking values in the set $\{0,1\}$. The recursion describes the ``vote" of agent $i$ at time $k$, where each agent either adopts the vote indicated by the majority of its inbound neighbors with some probability {$1-\epsilon$}, or chooses the opposite vote with probability {$\epsilon$}. Its dynamics are given by
   \begin{align}\label{Ex: Noisy Voter}
        R_i^{(k+1)} = \left( \zeta_i^{(k)} + 1\left(\sum_{j \rightarrow i} R_j^{(k)} \geq \frac{d_i^{-}}{2}\right)\right) \mod 2,
    \end{align}
where $\zeta^{(k)}_i~\sim\text{Ber}(\epsilon)$. In \cite{Aldo_Band_05, Alsmeyer_2011} the recursion is described for directed 3-regular trees, and its invariant distributions are described in terms of $\epsilon$. This model, defined on the square lattice, has also been studied via Monte Carlo simulation \cite{de1992isotropic}.  Note that removing the noise $\zeta_i^{(k)}$ from  \eqref{Ex: Noisy Voter} is equivalent to modeling the Glauber dynamics at zero temperature \cite{castellano2006zero}. Recursion \eqref{Ex: Noisy Voter} has also been applied to the study of opinion dynamics \cite{vilela2017majority} and trader dynamics on financial networks \cite{vilela2019majority}.
\item \textit{Google's PageRank algorithm:} This system of equations computes the rank of vertices on a directed graph of size $n$ by assigning to each vertex $i$ a rank $ p_i = r_i/n$, where  
    \begin{align*} 
         r_i  = (1-c)nq_i + \sum_{j \rightarrow i}\frac{c}{d_j^{+}} r_j,
    \end{align*}
where $r_i$ is the scale-free rank of vertex $i$, $q_i$ represents a personalization value, with ${\bf q} = (q_1, \dots, q_n)$ a probability vector, and $c \in (0,1)$ is a damping factor. The PageRank algorithm is one of the most popular centrality measures on networks, and is used in many different areas of science and engineering, including  word sense disambiguation \cite{agirre2014random}, spam detection \cite{gyongyi2004combating}, citation ranking \cite{chen2007finding}, visual search \cite{jing2008visualrank} and many others \cite{gleich2015pagerank}. It was originally created by Brin and Page \cite{Brin_Page_98} to rank the webpages in the WWW. {The distribution of the PageRank of a typical vertex on various random graphs has been shown to converge} to the attracting endogenous solution of a smoothing transform {in \cite{ Lee_Olv_20, Olvera_20}}, and this characterization has been used to prove the so-called ``power-law hypothesis" \cite{Volk_Litv_10, Jel_Olv_10,Olvera_12,Grav_etal_20}. 
\item \textit{A model for financial cascades:} This {system of equations} was proposed in \cite{deo2019limiting} for studying banking networks
    \begin{align}\label{Ex: Finance}
        r_i  = q_i + \sum_{j \rightarrow i}\frac{(r_j-v_j)^{+}\wedge b_j}{d_j^{+}},
    \end{align}
  where $r_i$ is the total wealth held by bank $i$, and $(q_i, b_i, v_i)$ represent its total external assets, its inter-bank loans, and its external liability, respectively. The work in \cite{deo2019limiting} studies only an Erd\H os-R\'enyi graph, but can also fit our more general framework. 
\end{enumerate}
\end{examples}

 The main assumption that we will make on the map $\Phi$ throughout the paper, is that it be Lipschitz continuous in its recursive arguments, i.e., the $\{R_j^{(k)}\}$, and continuous in its vertex attributes, since the graphs themselves will be changing when we take the large graph limit. Intuitively, this assumption ensures that we can control consecutive iterations of the recursion through a linear map on $\mathbb{R}^n$. Since the vertex attributes $\{ {\bf a}_i\}$ are assumed to take values on a separable metric space $\mathcal{S}'$, we can choose a convenient metric $\rho'$ for our continuity assumptions below. Throughout the rest of the paper, for ${\bf x}, \mathbf{\tilde x} \in  \mathcal{S} = \mathbb{N} \times \mathbb{N} \times \mathcal{S}'$, where ${\bf x} = (d^-, d^+, {\bf a})$ and $\mathbf{\tilde x} = (\tilde d^-, \tilde d^+, {\bf \tilde a})$, we define the metric:
\begin{equation} \label{eq:Metric}
\rho( {\bf x}, \mathbf{\tilde x}) = |d^- - \tilde d^-| + |d^+ - \tilde d^+| + \rho'({\bf a}, \mathbf{\tilde a}).
\end{equation}
It follows that $\mathcal{S}$ is also a separable metric space equipped with the metric $\rho$.

The main assumptions on the map $\Phi$ are given below. The norm $\| \cdot \|_p$ applied to a vector denotes the $\ell_p$-norm, while applied to a matrix it denotes its induced operator norm.

\begin{assum}{[R]} \label{A.PhiMap}
There exist continuous functions $\sigma_-, \sigma_+, \beta: \mathcal{S} \to \mathbb{R}^+$ and $p \in [1, \infty)$ such that:
\begin{enumerate} \renewcommand{\labelenumi}{\arabic{enumi})}
\item For any ${\bf v}, \mathbf{\tilde v} \in \mathbb{R}^n$, 
\begin{align*}
&\left( E\left[  \left| \Phi\left( {\bf x}_i, \zeta^{(0)}_i, \{ v_j, \xi_{j,i}^{(0)}: j \to i\} \right)  -\Phi\left( {\bf x}_i, \zeta^{(0)}_i, \{ \tilde v_j , \xi_{j,i}^{(0)}: j \to i\} \right)   \right|^p \right] \right)^{1/p} \\
&\leq \sigma_-({\bf x}_i ) \sum_{j \to i}  \left| v_j - \tilde v_j \right|{.}
\end{align*}
\item For any ${\bf x} \in \mathcal{S}$ and $r, \tilde r \in \mathbb{R}$,
$$|g(r, {\bf x}) - g(\tilde r, {\bf x}) | \leq \sigma_+({\bf x}) |r - \tilde r|{.}$$
\item For any ${\bf v} \in \mathbb{R}^n$, 
$$\left( E \left[  \left| \Phi\left( {\bf x}_i, \zeta^{(0)}_i, \{v_j, \xi_{j,i}^{(0)}: j \to i\} \right)    \right|^p \right] \right)^{1/p} \leq \sum_{j \to i} \sigma_-({\bf x}_i ) |v_j| + \beta( {\bf x}_i ){.}$$
\item One of the following holds: 
\begin{itemize}
\item[i.] The matrices $C$ and $C^{(0)}$ whose $(i,j)$th components are $C_{i,j} = \sigma_-({\bf x}_i) \sigma_+({\bf x}_j) 1( j \to i)$ and $C_{i,j}^{(0)} = \sigma_-({\bf x}_i) |g(0,{\bf x}_j )| 1(j \to i)$, respectively, satisfy $\| C \|_p \leq K < \infty$ and $\|C^{(0)} \|_p \leq K_0 < \infty$ for any {vertex-weighted} directed graph $G(V, E; \mathscr{A})$.
\item[ii.] There exists $K < \infty$ such that if $\text{supp}(\mu_0) \subseteq [-K, K]$, then $\| {\bf R}^{(1)} \|_\infty \leq K$ for any {vertex-weighted} directed graph $G(V, E; \mathscr{A})$. 
\end{itemize}
\item  There exist finite constants $H, Q, \alpha, \gamma > 0$ such that for $\mathbf{x} = (d^-, d^+, \mathbf{a})  \in \mathcal{S}$,  $\{ v_j : 1 \leq j \leq d^- \} \subseteq \mathbb{R}$, $r \in \mathbb{R}$, and any $\mathbf{\tilde x} \in \mathcal{S}$ such that $\rho(\mathbf{x}, \mathbf{\tilde x}) \leq \epsilon \in (0,1)$, 
\begin{align*}
& \left( E\left[ \left| \Phi\left( \mathbf{x}, \zeta, \left\{ v_j, \xi_j: 1 \leq j \leq d^- \right\} \right) - \Phi\left( \mathbf{\tilde x}, \zeta, \left\{ v_j, \xi_j: 1 \leq j \leq d^- \right\} \right) \right|^p \right] \right)^{1/p} \\
&\leq H \left( 1 \vee \left( \sum_{j \to i} \sigma_-(\mathbf{x})|v_j| + \beta(\mathbf{x}) \right) \right) \rho(\mathbf{x}, \mathbf{\tilde x})^\alpha,
\end{align*}
and 
$$| g(r, \mathbf{x}) - g(r, \mathbf{\tilde x})| \leq Q \left( 1 \vee \sigma_+(\mathbf{x}) |r| \right) \rho(\mathbf{x}, \mathbf{\tilde x})^\gamma.$$
\end{enumerate}
\end{assum}

\begin{examples}
All the models in Example~\ref{Ex:Models} satisfy Assumption~\ref{A.PhiMap} for different choices of $p, g, \sigma_-, \sigma_+$ and $\beta$, as described below:
\begin{enumerate}[leftmargin=5mm] \renewcommand{\labelenumi}{\arabic{enumi}.} 
\item  \textit{Generalized DeGroot model:} We can take any $p \geq 1$ and $g(r,{\bf x}) = r$, $\sigma_{-}({\bf x}_i) = (1-c)/d_i^{-}$, $\sigma_{+}({\bf x}_j) = 1$, and $\beta({\bf x}_i) = c \left( E[|f(q_i,\zeta_i)|^p] \right)^{1/p}$. In addition, if $|f(q,\zeta)| \leq K$ for some $K < \infty$, then, $\| {\bf R}^{(1)} \|_\infty \leq K$, and therefore, Assumption~\ref{A.PhiMap}(4)(ii) is satisfied. 
\item \textit{Noisy majority voter model:} We can take any $p \geq 1$ and $g(r, {\bf x}) = r$. To identify $\sigma_-, \sigma_+$ and $\beta$, note that for any vector ${\bf r}, {\bf \tilde r} \in \mathbb{R}^n$,
    \begin{align*}
        &\left( E\left[  \left| \Phi\left( {\bf x}_i, \zeta^{(0)}_i, \{r_j, \xi_{j,i}^{(0)}: j \to i\} \right)  -\Phi\left( {\bf x}_i, \zeta^{(0)}_i, \{\tilde r_j, \xi_{j,i}^{(0)}: j \to i\} \right)   \right|^p \right] \right)^{1/p} \\
        & = \left| 1\left(\frac{1}{d_i^-}\sum_{j \rightarrow i}r_j\geq \frac{1}{2}\right) - 1\left(\frac{1}{d_i^-}\sum_{j \rightarrow i}\tilde r_j\geq \frac{1}{2}\right) \right| \leq \frac{2}{d_i^-} \sum_{j \rightarrow i} |r_j - \tilde r_j |, 
    \end{align*}
so we can take $\sigma_-({\bf x}_i) = 2/d_i^{-}$ and $\sigma_{+}({\bf x}_j) = 1$. In addition, by choosing $p = 1$ and ignoring the $\text{mod } 2$ in the recursion we obtain $\beta({\bf x}_i) = P( \zeta = 1)^{1/p}$.  Since ${\bf R}_i^{(k)} \in \{0, 1\}$ for all $i \in V$, then Assumption~\ref{A.PhiMap}(4)(ii) is satisfied. 
\item \textit{Google's PageRank algorithm:} We can take $p = 1$, $g(r, {\bf x}_j) = c r/d_j^+$, $\sigma_{-}({\bf x}_i) = 1$, $\sigma_{+}( {\bf x}_j) =c/d_j^{+}$, and $\beta( {\bf x}_i) = (1-c)nq_i$.  The corresponding matrix $C$ satisfies  $\| C \|_1 \leq c < 1$. 
\item \textit{A model for financial cascades:} We can take $p = 1$ and $g(r,  {\bf x}_j) = ((r-v_j)^+ \wedge b_j)/d_j^+$.   Using the inequality $|(x\vee a)\wedge b-(y\vee a)\wedge b| \leq |x-y|$, gives that we can set $\sigma_{-}({\bf x}_i) = 1, \sigma_{+}( {\bf x}_j) = 1/ d_j^{+}$, and $\beta({\bf x}_i) = |q_i|$. The corresponding matrix $C$ satisfies $\|C \|_1 = 1$. 
\end{enumerate}
\end{examples}

\subsubsection{Analysis on a random graph} \label{SS.BothRandom}

 We now turn our attention to the analysis of the Markov chain $\{ \mathbf{R}^{(k)}: k \geq 0\}$ on a random graph $G(V_n, E_n; \mathscr{A}_n)$ constructed using the vertex attributes or latent variables $\mathscr{A}_n$. Formally, our model is defined on a probability space large enough to construct the sequence of latent variables $\{ \mathscr{A}_n: n \geq 1 \}$, random numbers to generate the random graph sequence $\mathbf{G} = \{ G(V_n, E_n; \mathscr{A}_n): n \geq 1\}$, the noises $\{\boldsymbol{\zeta}^{(k)}, \boldsymbol{\xi}^{(k)}: k \geq 0 \}$ (the same noises can be used for all graphs in $\mathbf{G}$), and additional independent random numbers for constructing couplings with the local weak limit of $\mathbf{G}$. 

We will now separate the different levels of randomness involved in the construction of $G(V_n, E_n; \mathscr{A}_n)$ and the corresponding Markov chain $\{ \mathbf{R}^{(k,n)}: k \geq 0\}$ for each $n \geq 1$; in the sequel we omit the dependence on $n$ and simply write $\mathbf{R}^{(k)} = \mathbf{R}^{(k,n)}$. The first level of randomness corresponds to the latent variables, or vertex attributes, $\{\mathscr{A}_n: n \geq 1\}$, which we identify through the sigma algebra: 
$$\mathscr{F}_n = \sigma( \mathscr{A}_n), \qquad n \geq 1,$$
which we use to define the conditional probability $\mathbb{P}_n(\,\cdot\,) = E[1(\cdot) | \mathscr{F}_n]$ and its corresponding conditional expectation $\mathbb{E}_n[\,\cdot\,] = E[ \,\cdot\, | \mathscr{F}_n]$.

Once the random graph $G(V_n, E_n; \mathscr{A}_n)$ is realized we construct the Markov chain $\{ \mathbf{R}^{(k)}: k \geq 0\}$ as described in Section~\ref{SS.FixedGraph}. The noises used to define the Markov chain are assumed to be independent of the latent variables $\mathscr{A}_n$ and of the realized graph $G(V_n, E_n; \mathscr{A}_n)$. When conditioning on the graph we will use the sigma algebra:
$$\mathscr{G}_n = \sigma\left( G(V_n, E_n; \mathscr{A}_n) \right),$$
which we use to define the conditional probability $\mathbf{P}_n(\,\cdot\,) = E[1(\cdot) | \mathscr{G}_n]$ and its corresponding conditional expectation $\mathbf{E}_n[\,\cdot\,] = E[\,\cdot\, | \mathscr{G}_n]$. Note that conditionally on $\mathscr{G}_n$ the only remaining source of randomness in $\{\mathbf{R}^{(k)}: k \geq 0\}$ is the initial state $\mathbf{R}^{(0)}$ and noise sequences $\{\boldsymbol{\zeta}^{(k)}, \boldsymbol{\xi}^{(k)}: k \geq 0\}$. To emphasize the random nature of the graph we now use upper case letters for denoting the vertex attributes and their full marks, i.e., $\mathbf{A}_i$ denotes the attribute or latent variable of vertex $i \in V_n$, while $\mathbf{X}_i = (D_i^-, D_i^+, \mathbf{A}_i)$ denotes its full mark. Note that the $\{{\bf X}_i: i \in V_n\}$ are measurable with respect to $\mathscr{G}_n$, but not necessarily $\mathscr{F}_n$. 

Our main goal in this paper is to show that the marginals of the stationary distribution of the Markov chain $\{ \mathbf{R}^{(k)}: k \geq 0\}$ on a {vertex-weighted} directed graph $G(V_n, E_n; \mathscr{A}_n)$ can be characterized via a branching distributional fixed-point equation whenever the graph sequence $\mathbf{G}$ converges in the local weak sense to a marked Galton-Watson process. However, the existence of a stationary distribution for $\{ \mathbf{R}^{(k)}: k \geq 0\}$ for a fixed graph is unrelated to the random graph model used to construct the graph.  In most of the examples that we have in mind, the existence of a stationary distribution for $\{ {\bf R}^{(k)}: k \geq 0\}$ taking values in $\mathbb{R}^n$ can be established by showing that the map $\Phi$ defines a contraction under a suitable Wasserstein metric. Since it is possible to establish the existence of a stationary distribution in other ways, our main theorem (Theorem~\ref{T.Main}) will focus on the analysis of finitely many iterations, i.e., $\{ {\bf R}^{(r)}: 0 \leq r \leq k\}$ for any fixed $k$, without specifically imposing conditions to ensure its convergence as $k \to \infty$. The result for the graph when $\Phi$ defines a contraction on $\mathbb{R}^n$ is given as a corollary to the main theorem.

\section{Characterizing the typical behavior} \label{S.Characterization}

We focus now on characterizing the distribution of a uniformly chosen component of the vector ${\bf R}^{(k)}$, since it represents the typical behavior of a vertex in the graph $G(V_n, E_n; \mathscr{A}_n)$ under the iterations of the map $\Phi$. For graphs with exchangeable vertices, this also corresponds to the common marginal distribution of its vertices. Throughout the paper we denote this random variable by $R^{(k)}_I$, where $I$ is a uniformly chosen index in $V_n$. The main idea behind our characterization is that, provided the graph $G(V_n, E_n; \mathscr{A}_n)$ converges in the local weak sense to a marked Galton-Watson process, $R^{(k)}_I$ will satisfy a recursion on a tree which will converge, as $n, k \to \infty$, to a random variable that can be written in terms of a solution to a branching distributional fixed-point equation. 

We start by defining a delayed marked Galton-Watson process. Nodes in a tree will have labels of the form $\mathbf{i} = (i_1, \dots, i_k) \in \mathcal{U}$, where $\mathcal{U} = \bigcup_{k=0}^\infty \mathbb{N}_+^k$ with the convention that $\mathbb{N}_+^0 := \{ \emptyset\}$ denotes the root. To simplify the notation, we also write $\mathbf{i} = i$ instead of $\mathbf{i} = (i)$ for labels of length one. The index concatenation operation is denoted by $(\mathbf{i},j) =(i_1, \dots, i_k, j)$. 

\begin{defn} \label{D.markGW}
A {{\em delayed marked Galton-Watson tree}} with marks on {some} separable metric space $M$, is a tree constructed using a {family} of independent random vectors ${\boldsymbol{\mathcal{X}}=} \{ (\mathcal{N}_\mathbf{i}, {\boldsymbol{\mathcal{Y}}_\mathbf{i}}): \mathbf{i} \in \mathcal{U} \}$, such that $\mathcal{N}_\mathbf{i} \in \mathbb{N}$, ${\boldsymbol{\mathcal{Y}}_\mathbf{i}} \in M$, and the vectors $\{ (\mathcal{N}_\mathbf{i}, {\boldsymbol{\mathcal{Y}}_\mathbf{i}}): \mathbf{i} \in \mathcal{U}, \mathbf{i} \neq \emptyset \}$ are i.i.d. The {\em delay} refers to the possibility of $(\mathcal{N}_\emptyset, {\boldsymbol{\mathcal{Y}}_\emptyset})$ having a different distribution. The root of the tree, labeled $\emptyset$, constitutes generation zero, $\mathcal{A}_0 = \{ \emptyset\}$. Generation $k$, for $k \geq 1$, is determined recursively via
$$\mathcal{A}_k = \{ (\mathbf{i}, j): \mathbf{i} \in \mathcal{A}_{k-1}, \, 1 \leq j \leq \mathcal{N}_\mathbf{i}\}.$$
We denote the unmarked tree $\mathcal{T} = \bigcup_{k=0}^\infty \mathcal{A}_k$, and the marked one $\mathcal{T}({\boldsymbol{\mathcal{X}}}) = \{ {(\mathcal{N}_\mathbf{i}, \boldsymbol{\mathcal{Y}}_\mathbf{i})}: \mathbf{i} \in \mathcal{T} \}$; their restrictions to the first $k$ generations are denoted $\mathcal{T}^{(k)} = \bigcup_{r=0}^k \mathcal{A}_r$ and $\mathcal{T}^{(k)}({\boldsymbol{\mathcal{X}}}) = \{ {(\mathcal{N}_\mathbf{i}, \boldsymbol{\mathcal{Y}}_\mathbf{i})}: \mathbf{i} \in \mathcal{T}^{(k)} \}$, respectively.
\end{defn}

For a directed graph $G(V_n, E_n; \mathscr{A}_n)$, let $I \in V_n$ be a vertex uniformly chosen at random, and let $A_k$ denote the set of vertices $j \in V_n$ having a directed path of length $k$ connecting them to vertex $I$, with $A_0 = \{ I\}$. Define $\mathcal{G}_I^{(k)} =\bigcup_{r=0}^k A_r$ to be the subgraph of $G(V_n, E_n; \mathscr{A}_n)$ whose vertex set is $\bigcup_{r=0}^k A_r$. If the vertices in the graph $G(V_n, E_n; \mathscr{A}_n)$ have marks $\{ \mathbf{X}_i: i \in V_n\}$, then we use $\mathcal{G}_I^{(k)}(\mathbf{X})$ to denote the graph $\mathcal{G}_I^{(k)}$ including the marks. 

Our main result is based on a coupling of the exploration of the inbound component of a uniformly chosen vertex $I$ in a {vertex-weighted directed} graph $G(V_n, E_n; \mathscr{A}_n)$ and a delayed marked Galton-Watson process. The exact description of this coupling is given in the following definitions.

\begin{defn}
We say that two simple {directed graphs} $G(V,E)$ and $G'(V', E')$ are {\em isomorphic} if there exists a bijection $\theta: V \to V'$ such that edge $(i,j) \in E$ if and only if edge $(\theta(i), \theta(j)) \in E'$.
We say that two directed multigraphs $G(V,E)$ and $G'(V',E')$ are {\em isomorpic} if there exists a bijection $\theta: V \to V'$ such that $l(i) = l(\theta(i))$ and $e(i,j) = e(\theta(i), \theta(j))$ for all $i \in V$ and all $(i,j) \in E$, where $l(i)$ is the number of self-loops of vertex $i$ and $e(i,j)$ is the number of edges from vertex $i$ to vertex $j$. In both cases,  we write $G \simeq G'$. 
\end{defn}

\begin{defn}
A graph $G(V, E; \mathscr{A})$ is called a {\em vertex-weighted} {directed} graph if each of its vertices $i \in V$ has an attribute (weight) $\mathbf{a}_i \in \mathcal{S}$ assigned to it, where $\mathcal{S}$ is a separable metric space, and  $\mathscr{A} = (\mathbf{a}_i)_{i \in V}$ are the vertex attributes.
\end{defn}

Our main result is stated for {directed} graphs for which a strong coupling of its local neighborhoods, as defined below, exists. The work in \cite{Olvera_21} shows that such a coupling exists for any random graph generated via either a directed configuration model or any of the graphs in the family of inhomogeneous random digraphs with rank-1 kernels, which includes the directed versions of well-known models such as the Erd\H os-R\'enyi graph, the Chung-Lu model, the Norros-Reittu model, or the generalized random graph. For completeness, we include in the {appendix} a brief description of these two families of random graphs. With the exception of the Erd\H os-R\'enyi graph, all these models are build using latent variables.

\begin{defn} \label{D.StrongCoupling}
Let $\boldsymbol{G}=\{ G(V_n, E_n; \mathscr{A}_n): n \geq 1\}$ be a sequence of vertex-weighted directed random graphs where each $G(V_n, E_n; \mathscr{A}_n)$ is generated using the latent variables in $\mathscr{A}_n$; let $\mathscr{F}_n = \sigma(\mathscr{A}_n)$ and define $\mathbb{P}_n(\,\cdot\,) = P(\,\cdot\,| \mathscr{F}_n)$ and $\mathbb{E}_n[\,\cdot\,] = E[\,\cdot\, | \mathscr{F}_n]$. We say that $\mathbf{G}$ admits a {\em strong coupling} with a marked Galton-Watson tree $\mathcal{T}(\boldsymbol{\mathcal{X}})$ if for each $i \in V_n$ there exists a marked Galton-Watson tree $\mathcal{T}_{\emptyset(i)}(\boldsymbol{\mathcal{X}})$ such that if $I$ is uniformly chosen in $V_n$, independent of $\mathscr{G}_n$,  then $\mathcal{T}_{\emptyset(I)}(\boldsymbol{\mathcal{X})} \stackrel{\mathcal{D}}{=} \mathcal{T}(\boldsymbol{\mathcal{X}})$, and such that for any $\epsilon \in (0,1)$ and any $k \geq 1$ we have that the event 
$$C_i^{(k,\epsilon)} = \left\{ \bigcap_{\mathbf{j} \in \mathcal{T}^{(k)}_{\emptyset(i)} } \{ \rho( \mathbf{X}_{\theta_i(\mathbf{j})}, \boldsymbol{\mathcal{X}}_{\mathbf{j}} ) \leq \epsilon \}, \, \mathcal{G}_i^{(k)} \simeq \mathcal{T}^{(k)}_{\emptyset(i)}  \right\},$$
where $\theta_i$ is the bijection defining $\mathcal{G}_i^{(k)} \simeq \mathcal{T}^{(k)}_{\emptyset(i)}$, satisfies
\begin{align*}
\mathbb{E}_n\left[ \rho(\mathbf{X}_{I}, \boldsymbol{\mathcal{X}}_{\emptyset(I)}) \right] &= \frac{1}{|V_n|} \sum_{i\in V_n} \mathbb{E}_n\left[ \rho(\mathbf{X}_i, \boldsymbol{\mathcal{X}}_{\emptyset(i)})\right] \xrightarrow{P} 0, \qquad n\to\infty \\
\intertext{and}
\mathbb{P}_n\left( C_{I}^{(k,\epsilon)}\right) &= \frac{1}{|V_n|} \sum_{i\in V_n} \mathbb{P}_n\left( C_i^{(k,\epsilon)} \right) \xrightarrow{P} 1, \qquad n\to\infty
\end{align*}

Furthermore, for any fixed $m \geq 1$ and $\{I_j: 1 \leq j \leq m\}$ i.i.d.~random variables uniformly chosen in $V_n$, independent of $\mathscr{G}_n$, there exist i.i.d. Galton-Watson trees denoted $\left\{ \mathcal{T}_{\emptyset(I_j)}(\boldsymbol{\mathcal{X}}): 1 \leq j \leq m\right\}$ distributed as $\mathcal{T}(\boldsymbol{\mathcal{X}})$, whose roots correspond to the vertices $\{I_j: 1 \leq j \leq m\}$ in $G(V_n, E_n;\mathscr{A}_n)$, such that

$$\sum_{j=1}^m \mathbb{E}_n\left[ \rho(\mathbf{X}_{I_j}, \boldsymbol{\mathcal{X}}_{\emptyset(I_j)}) \right] \xrightarrow{P} 0 \qquad \text{and} \qquad \mathbb{P}_n\left( \bigcap_{j=1}^m C_{I_j}^{(k,\epsilon)} \right) \xrightarrow{P} 1, \qquad n \to \infty.$$
\end{defn}

Let $M$ be a Polish space and let  $\eta$ be its corresponding distance. Fix $p \in [1, \infty)$ and let $\mathcal{P}_p(M)$ be the set of distributions on $M$ with finite $p$th moment. We use $d_p$ to denote the Wasserstein metric of order $p$ given by
$$d_p(\mu, \nu) = \inf\left\{ \left( E\left[ \eta(X,Y)^p \right] \right)^{1/p}: \text{law}(X) = \mu, \, \text{law}(Y) = \nu \right\},$$
where the infimum is taken over all couplings of the random variables/vectors $X$ and $Y$. When $M=\mathbb{R}$ and $\eta(x,y) = |x-y|$, the $d_p$ metric admits the explicit representation:
$$d_p(\mu, \nu)^p = \int_0^1 | F^{-1}(u) - G^{-1}(u) |^p du,$$
for $F(x) = \mu((-\infty, x])$, $G(x) = \nu((-\infty, x])$, and $f^{-1}(u) = \inf \{ x \in \mathbb{R}: f(x) \geq u \}$ the generalized inverse of function $f$. 
Throughout the paper the metric $d_p$ will vary depending on the underlying space. When working on $\mathbb{R}$ we will use the explicit representation stated above, when working on $\mathcal{S}$ we use $\eta = \rho$, and when working on $\mathbb{R}^n$ we take $\eta$ to be the corresponding $\ell_p$ metric. The underlying space will become clear from the context.

The main results in this paper will require that we use random measures, sometimes measurable with respect to $\mathscr{F}_n$ and sometimes with respect to $\mathscr{G}_n$. In all cases, the limiting measure is non-random, and the convergence occurs in probability, i.e.,
$$d_p(\nu_n, \nu) \xrightarrow{P} 0, \qquad n \to \infty.$$ 
Note that if $\nu_n$ is $\mathscr{F}_n$-measurable ($\mathscr{G}_n$-measurable), then so is $d_p(\nu_n,\nu)$. This follows from the measurability of the of optimal coupling (see Corollary 5.22 in \cite{villani2008optimal}).

\begin{remark}
Note that Definition~\ref{D.StrongCoupling} implies that $\mathbf{X}_I \xrightarrow{P} \boldsymbol{\mathcal{X}}_\emptyset$ as $n \to \infty$, which in turn implies that if $\upsilon_n^*(\cdot) = \mathbb{P}_n\left(\mathbf{X}_I \in \cdot \right)$ and $\upsilon^*(\cdot) = P\left( \boldsymbol{\mathcal{X}}_\emptyset \in \cdot \right)$, then $d_1(\upsilon_n^*, \upsilon^*) \xrightarrow{P} 0$ as $n \to \infty$. On the other hand, if $\mathbf{X}_\xi$ denotes the mark of the inbound neighbor of vertex $I$ coupled to node 1 on the limiting tree, we have $\mathbf{X}_\xi \xrightarrow{P} \boldsymbol{\mathcal{X}}_1$ as $n \to \infty$; however, if $\upsilon_n(\cdot) = \mathbb{P}_n\left(\mathbf{X}_\xi \in \cdot \right)$ and $\upsilon(\cdot) = P\left( \boldsymbol{\mathcal{X}}_1 \in \cdot \right)$, we may not have that $d_1(\upsilon_n, \upsilon) \xrightarrow{P} 0$ as $n \to \infty$.  The weaker mode of convergence for $\mathbf{X}_\xi$, is due to the size-bias introduced during the exploration, which in some cases could lead to $E[\rho({\boldsymbol{\mathcal{X}}_1}, \mathbf{x}_0)] = \infty$ for any fixed $\mathbf{x}_0 \in \mathcal{S}$. 
\end{remark}

\begin{assum}{[G]} \label{A.RGmodels}
The graph {sequence $\mathbf{G} = \{G(V_n, E_n; \mathscr{A}_n): n \geq 1\}$} admits a strong coupling with a (delayed) marked Galton-Watson process, and for the same $p \in [1, \infty)$ from Assumption~\ref{A.PhiMap}, the following hold:
\begin{itemize}
\item The expectations 
\begin{align*}
&E\left[ (\mathcal{N}_\emptyset \sigma_-({\boldsymbol{\mathcal{X}}_\emptyset}))^p + \beta({\boldsymbol{\mathcal{X}}_\emptyset})^p\right] \qquad \text{and} \\
&E\left[ (\mathcal{N}_1\sigma_-({\boldsymbol{\mathcal{X}}_1}) \sigma_+({\boldsymbol{\mathcal{X}}_1}))^p+ (\sigma_+({\boldsymbol{\mathcal{X}}_1}) \beta({\boldsymbol{\mathcal{X}}_1}))^p + \sigma_+({\boldsymbol{\mathcal{X}}_1})^p +   |g(0, {\boldsymbol{\mathcal{X}}_1})|^p \right]
\end{align*}
are finite. 
\item For $\mathbf{x} = (d^-, d^+, \mathbf{a}) \in \mathcal{S}$ and $f_*(\mathbf{x}) = \sigma_-(\mathbf{x}) d^-$, the following limits hold as $n \to \infty$:
$$\frac{1}{|V_n|} \sum_{i\in V_n} f_*(\mathbf{X}_i)^p \xrightarrow{P} E\left[ f_*(\boldsymbol{\mathcal{X}}_\emptyset)^p \right], \qquad \frac{1}{|V_n|} \sum_{i\in V_n} \beta(\mathbf{X}_i)^p \xrightarrow{P} E\left[ \beta(\boldsymbol{\mathcal{X}}_\emptyset)^p\right].$$
\end{itemize}
\end{assum}

We are now ready to state our main theorem, whichgives the existence of a coupling between the trajectory of the process $\{ R_I^{(r)}: 0 \leq r \leq k\}$ along a uniformly chosen vertex $I \in V_n$, and the trajectory of the root node in an equivalent process constructed on the limiting tree. In the statement of the theorem, the random vector $(\zeta, \{\xi_j: j \geq 1\})$ denotes a generic noise vector, and is assumed to be independent of ${\boldsymbol{\mathcal{X}}_\emptyset}$ and ${\boldsymbol{\mathcal{X}}_1}$.

\begin{theorem} \label{T.Main}
Suppose the map $\Phi$ satisfies Assumption~\ref{A.PhiMap}, the directed graph sequence $\boldsymbol{G} = \{G(V_n,E_n; \mathscr{A}_n): n \geq 1\}$ satisfies Assumption~\ref{A.RGmodels}. Let $\mu_0(\cdot) = P(R_1^{(0)} \in \cdot)$ satisfy $E\left[ |R^{(0)}_1|^p \right] < \infty$ if Assumption~\ref{A.PhiMap}(4)(i) holds, or $\text{supp}(\mu_0) \subseteq [-K, K]$ if Assumption~\ref{A.PhiMap}(4)(ii) holds. 
\begin{enumerate}[label=\Alph*), leftmargin=*]
\item Let $\mu_{k,n}(\cdot) = \mathbb{P}_n\big( R_I^{(k)} \in \cdot \big)$, where $I$ is uniformly distributed in $\{1, 2, \dots, n\}$, independent of $\mathscr{G}_n$. 
Then, for any fixed $k \geq 0$ there exists a sequence of random variables $\{ \mathcal{R}_\emptyset^{(0)}, \mathcal{R}_\emptyset^{(1)}, \dots, \mathcal{R}_\emptyset^{(k)} \}$ constructed on the coupled marked Galton-Watson tree $\mathcal{T}_{\emptyset(I)}^{(k)}(\boldsymbol{\mathcal{X}})$ from Assumption~\ref{A.RGmodels}, on the same probability space as $\{ R_I^{(0)}, R_I^{(1)}, \dots, R_I^{(k)} \}$, such that
$$\max_{0 \leq r \leq k} \mathbb{E}_n\left[ \left| R_I^{(r)} - \mathcal{R}_\emptyset^{(r)} \right|^p \right] \xrightarrow{P} 0, \qquad n \to \infty.$$
In particular, {if $\nu_k(\cdot) = P\left( \mathcal{R}_\emptyset^{(k)} \in \cdot \right)$, then, for any fixed $k \geq 0$,
$$d_p(\mu_{k,n}, \nu_k) \xrightarrow{P} 0, \qquad n \to \infty.$$}
\item Furthermore, for any fixed $m,k \geq 1$, $\{I_j: 1 \leq j \leq m\}$ i.i.d.~uniformly chosen in $V_n$,  independent of $\mathscr{G}_n$, and for any set of bounded and continuous functions on $\mathbb{R}^{k+1}$, $\{ f_j: 1 \leq j \leq m\}$, we have
$$E\left[ \prod_{j=1}^m f_j(R_{I_j}^{(0)}, \dots, R_{I_j}^{(k)}) \right] \to \prod_{j = 1}^m E[f_j(\mathcal{R}_\emptyset^{(0)}, \dots, \mathcal{R}_\emptyset^{(k)})], \qquad n \to \infty.$$

\item Moreover, if $E[ (\mathcal{N}_1 \sigma_-({\boldsymbol{\mathcal{X}}_1})\sigma_+({\boldsymbol{\mathcal{X}}_1}))^p] =: c^p < 1$, then there exists a probability measure $\nu$ such that 
$$d_p(\nu_k, \nu) \to 0 \quad \text{a.s.}, \qquad k \to \infty,$$
where $\nu$ is the probability measure of a random variable $\mathcal{R}^*$ that satisfies:
\begin{align*}
\mathcal{R}^* &= \Phi\left( {\boldsymbol{\mathcal{X}}_\emptyset}, \zeta, \{ \mathcal{V}_j, \xi_j: 1 \leq j \leq \mathcal{N}_\emptyset \} \right), 
\end{align*}
with the $\{\mathcal{V}_j\}$ i.i.d.~copies of $\mathcal{V}$, independent of $({\boldsymbol{\mathcal{X}}_\emptyset}, \zeta, \{ \xi_j: j \geq 1\})$, {$\boldsymbol{\mathcal{X}}_\emptyset$ independent of $(\zeta, \{ \xi_j: j \geq 1\})$, and} $\mathcal{V}$ the attracting endogenous solution to the distributional fixed-point equation:
\begin{align*}
\mathcal{V} \stackrel{\mathcal{D}}{=} \Psi\left( {\boldsymbol{\mathcal{X}}_1}, \zeta, \{ \mathcal{V}_j, \xi_j: 1 \leq j \leq \mathcal{N}_1 \} \right),
\end{align*}
where the $\{ \mathcal{V}_j \}$ are i.i.d.~copies of $\mathcal{V}$, independent of $({\boldsymbol{\mathcal{X}}_1}, \zeta, \{ \xi_j: j \geq 1\})$, and {$\boldsymbol{\mathcal{X}}_1$ is independent of $(\zeta, \{ \xi_j: j \geq 1\})$}.
\end{enumerate}
\end{theorem}

{\begin{remark} \label{R.Observations}
The following two convergence in probability results follow from Theorem~\ref{T.Main}.
\begin{itemize}
    \item Part (A) implies that 
    $$\max_{0 \leq r \leq k} \left|R_I^{(r)} - \mathcal{R}_\emptyset^{(r)}\right|\xrightarrow{P} 0, \qquad n \to \infty,$$
    since the dominated convergence theorem gives:
    \begin{align*}
        E\left[\max_{0 \leq r\leq k}\left|R_I^{(r)} - \mathcal{R}_\emptyset^{(r)}\right|^p\wedge 1\right] \leq \sum_{r = 0}^{k} E\left[\mathbb{E}_n\left[\left|R_I^{(r)} - \mathcal{R}_\emptyset^{(r)}\right|^p\right]\wedge 1\right] \to 0, \qquad n \to \infty.
    \end{align*}
    \item By using the second moment method one can show that for any bounded and continuous function $f$ on $\mathbb{R}^{k+1}$,
$$\frac{1}{n}\sum_{i = 1}^n f(R_i^{(0)}, R_i^{(1)}, ..., R_i^{(k)})\xrightarrow{P} E[f(\mathcal{R}_\emptyset^{(0)}, \mathcal{R}_\emptyset^{(1)},...\mathcal{R}_\emptyset^{(k)})], \qquad n \to \infty.$$
\end{itemize}
\end{remark}}
As a corollary we obtain the convergence of $R_I$ provided $\Phi$ also defines a contraction on $\mathbb{R}^n$. For this result the entire graph $G(V_n, E_n; \mathscr{A}_n)$ has been realized and remains fixed, so we use the conditional probability $\mathbf{P}_n(\cdot)$. 

\begin{cor} \label{C.MainContraction}
Suppose that in addition to the assumptions of Theorem~\ref{T.Main},  under Assumption~\ref{A.PhiMap}(4)(i), we have $\| C \|_p \leq K < 1$ $P$-a.s.  Then, if $\lambda_{k,n}(\cdot) = \mathbf{P}_n( \mathbf{R}^{(k)} \in \cdot)$, there exists a probability measure $\lambda_n$ such that
$$d_p(\lambda_{k,n}, \lambda_n) \to 0 \quad \mathbf{P}_n\text{-a.s}, \qquad k \to \infty.$$
Moreover, if we let $\mathbf{R} = (R_1, \dots, R_n) \in \mathbb{R}^n$ be distributed according to $\lambda_n$, and define $\mu_n(\cdot) = \mathbf{P}_n(R_I \in \cdot)$, where $I$ is uniformly distributed in $\{1, 2,\dots, n\}$, then, 
$$d_p(\mu_n, \nu) \xrightarrow{P} 0, \qquad n \to \infty,$$
for $\nu$ the limiting measure in Theorem~\ref{T.Main}.
\end{cor}

In other words the corollary states that the limits in the following diagram commute
\begin{figure}
\[ \begin{tikzcd}
R_I^{(k,n)} \arrow{r}{n\to\infty} \arrow[swap]{d}{k\to\infty} & \mathcal{R}_\emptyset^{(k)} \arrow{d}{k\to\infty} \\%
R_I^{(\infty,n)} \arrow{r}{n\to\infty}& \mathcal{R}^*
\end{tikzcd}
\qquad \text{equivalently} \qquad
\begin{tikzcd}
\mu_{k,n} \arrow{r}{n\to\infty} \arrow[swap]{d}{k\to\infty} & \nu_k \arrow{d}{k\to\infty} \\%
\mu_n \arrow{r}{n\to\infty}& \nu
\end{tikzcd}
\]
\caption{Diagram representing the different limits, on the left in terms of the random variables and on the right in terms of their corresponding probability measures. In each diagram the quantities on the left are defined on the graph while quantities on the right are defined on the limiting tree.}
\end{figure}

\section{Solutions to branching recursions} \label{S.EndogenousSol}

We now explain what it means to be the attracting endogenous solution to a branching distributional fixed-point equation. To do this, consider a (non delayed) marked Galton-Watson process $\mathcal{T}({\boldsymbol{\mathcal{Z}}})$, as in Definition~\ref{D.markGW}. The marks in Theorem~\ref{T.Main} correspond to ${\boldsymbol{\mathcal{Z}}_\mathbf{i}} = ({\boldsymbol{\mathcal{X}}_\mathbf{i}}, \zeta_\mathbf{i}, \{ \xi_{(\mathbf{i},j)} \}_{j \geq 1} )$, where $(\zeta_\mathbf{i}, \{ \xi_{(\mathbf{i},j)} \}_{j \geq 1})$ is independent of ${\boldsymbol{\mathcal{X}}_\mathbf{i}}$ and has the same distribution as the noise vectors in the recursion.  We use ${\boldsymbol{\mathcal{Z}}} = ({\boldsymbol{\mathcal{X}}}, \zeta, \{ \xi_j\}_{j \geq 1})$ to denote a generic branching vector having the common distribution of the $\{ {\boldsymbol{\mathcal{Z}}_{\bf i}}: {\bf i} \in \mathcal{U} \}$. 

To construct a solution to the branching distributional fixed-point equation:
\begin{equation} \label{eq:limitSFPE}
\mathcal{V} \stackrel{\mathcal{D}}{=} \Psi\left( {\boldsymbol{\mathcal{X}}}, \zeta, \{ \mathcal{V}_j, \xi_j: 1 \leq j \leq \mathcal{N} \} \right),
\end{equation}
where the $\{ \mathcal{V}_j: j \geq 1\}$ are i.i.d.~copies of $\mathcal{V}$, independent of ${\boldsymbol{\mathcal{Z}}}$, let $\eta_0$ be a probability measure on $\mathbb{R}$, and define for any $k \geq 1$:
\begin{equation} \label{eq:LimitConstruc}
\mathcal{V}_\mathbf{i}^{(r)} = \Psi \left( {\boldsymbol{\mathcal{X}}_\mathbf{i}}, \zeta_\mathbf{i}, \{ \mathcal{V}_{(\mathbf{i},j)}^{(r-1)}, \xi_{(\mathbf{i},j)}: 1 \leq j \leq \mathcal{N}_\mathbf{i} \} \right), \qquad \mathbf{i} \in \mathcal{A}_{k-r}, 1 \leq r \leq k,
\end{equation}
and the $\{ \mathcal{V}_\mathbf{i}^{(0)}: \mathbf{i} \in \mathcal{A}_k \}$ i.i.d.~with common distribution $\eta_0$; let $\eta_k(\cdot) = P( \mathcal{V}_\emptyset^{(k)} \in \cdot \, )$. 

When the map $\Psi$ defines a contraction under a Wasserstein metric, i.e., when
$$d_p(\eta_{k+1}, \eta_k) \leq c d_p(\eta_k, \eta_{k-1}),$$
for some $0 < c < 1$ and some $p \in [1, \infty)$, we have that there exists a probability measure $\eta$ such that
$$d_p(\eta_k, \eta) \to 0, \qquad k \to \infty.$$
We call this $\eta$ the solution to \eqref{eq:limitSFPE} associated to the initial distribution $\eta_0$. 

\begin{defn}
Let $\mathcal{G}_\mathcal{T} = \sigma\left( {\boldsymbol{\mathcal{Z}}_{\mathbf{i}}}: \mathbf{i} \in \mathcal{T} \right)$. We say that that $\eta$ is an {\em endogenous} solution to the distributional fixed-point equation \eqref{eq:limitSFPE} if there exists a random variable $\mathcal{V}_\emptyset$, distributed according to $\eta$, such that $\mathcal{V}_\emptyset$ is $\mathcal{G}_{\mathcal{T}}$-measurable. 
\end{defn}

The solution in Theorem~\ref{T.Main} corresponds to choosing $\eta_0(\cdot) = P( g(0, {\boldsymbol{\mathcal{X}}_1}) \in \cdot \,)$, and it is endogenous since the $\{ g(0, {\boldsymbol{\mathcal{X}}_\mathbf{i}}): \mathbf{i} \in \mathcal{A}_k\}$ are measurable with respect to $\mathcal{G}_{\mathcal{T}}$ for all $k \geq 0$. We refer to this solution as the {\em attracting endogenous solution} to \eqref{eq:limitSFPE}. In general, equations like \eqref{eq:limitSFPE} can have multiple solutions  \cite{Als_Big_Mei_10, Alsm_Mein_10b, Biggins_98},  and in some cases, even multiple endogenous solutions \cite{Alsm_Mein_10b}. We refer the reader to \cite{Aldo_Band_05} and \cite{Mac_Stu_Swa_18} for a thorough discussion of the notion of {\em endogeny} and its characterization. 

We also point out that, in general, the existence of a solution $\mathcal{V}$ to the branching distributional fixed-point equation does not require the map $\Psi$ to define a contraction (i.e., $c \in (0,1)$), but the contraction approach is the easiest to state with the level of generality we aim for, since other approaches (e.g., stochastic monotonicity) may require more specific conditions on the initial distribution $\eta_0$.

\section{Proofs} \label{S.Proofs}

This section contains the proofs of Theorem~\ref{T.Main} and Corollary~\ref{C.MainContraction}. 

We start by constructing the random variables $\{ \mathcal{R}_\emptyset^{(0)}, \mathcal{R}_\emptyset^{(1)}, \dots, \mathcal{R}_\emptyset^{(k)} \}$ in Theorem~\ref{T.Main}. Fix $k \geq 1$, and choose $I \in V_n$. Let $\mathcal{G}_I^{(k)}(\mathbf{X})$ denote the marked subgraph obtained by exploring the depth-$k$ in-component of vertex $I$, and let $\mathcal{T}^{(k)}({\boldsymbol{\mathcal{X}}}) := \mathcal{T}^{(k)}_{\emptyset(I)}({\boldsymbol{\mathcal{X}}})$ be its strong coupling, as described in Definition~\ref{D.StrongCoupling}. We say that the unmarked coupling has been successful if
$$\mathcal{G}_I^{(k)} \simeq \mathcal{T}^{(k)},$$ 
in which case $\theta(\mathbf{i}) := \theta_{I}(\mathbf{i}) \in V_n$ denotes the identity of the vertex that corresponds to node $\mathbf{i}$ in the tree. 

Assume first that the unmarked coupling of $\mathcal{G}_I^{(k)}$ and $\mathcal{T}^{(k)}$ has been successful. Recall that noises on the graph $G(V_n, E_n;\mathscr{A}_n)$ are given by the sequence $\{ \zeta_i^{(r)}, \xi_{j,i}^{(r)}: i,j \in V_n, \, r \geq 0 \}$.  Define $\mathcal{R}_\emptyset^{(0)} = R_I^{(0)}$, and
$$\mathcal{R}_\emptyset^{(k)} = \Phi\left( {\boldsymbol{\mathcal{X}}_\emptyset}, \zeta_{I}^{(k-1)} , \{ \mathcal{V}_{j}^{(k-1)} , \xi_{\theta(j),I}^{(k-1)}: 1 \leq j \leq \mathcal{N}_\emptyset \} \right),$$
where
\begin{align*}
\mathcal{V}_\mathbf{i}^{(r)} &= \Psi\left( {\boldsymbol{\mathcal{X}}_\mathbf{i}}, \zeta_{\theta(\mathbf{i})}^{(k-r)}, \{ \mathcal{V}_{(\mathbf{i},j)}^{(r-1)}, \xi_{ \theta((\mathbf{i},j)), \theta(\mathbf{i})}^{(k-r)} : 1 \leq j \leq \mathcal{N}_\mathbf{i} \} \right), \qquad \mathbf{i} \in \mathcal{A}_{k-r}, \, 1 \leq r < k, \\
\mathcal{V}_\mathbf{i}^{(0)} &= g(R_{\theta(\mathbf{i})}^{(0)}, {\boldsymbol{\mathcal{X}}_\mathbf{i}}), \qquad \mathbf{i} \in \mathcal{A}_k.
\end{align*}
{Figure~\ref{F.Construction} depicts the construction for the case $k = 3$, which shows how the computation of $\mathcal{R}_\emptyset^{(k)}$ requires the entire marked tree $\mathcal{T}^{(k)}(\boldsymbol{\mathcal{X}})$.}

\begin{figure}[h]
\includegraphics[scale=0.95]{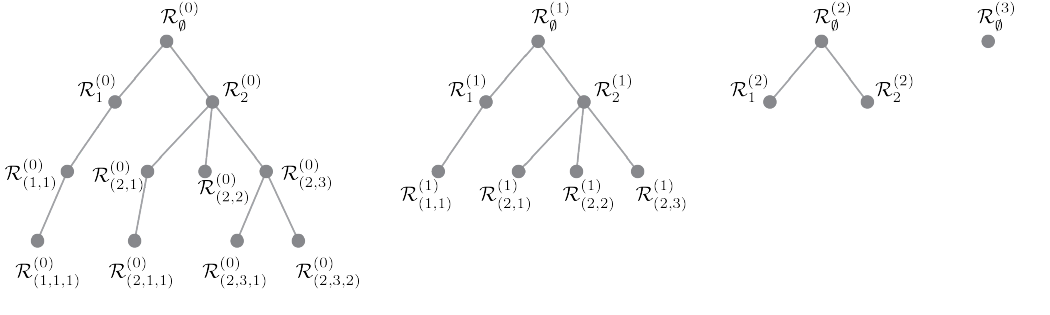}
\caption{Trajectory of the root on the limiting tree. The picture depicts the construction of $\{\mathcal{R}_\emptyset^{(0)}, \mathcal{R}_\emptyset^{(1)}, \mathcal{R}_\emptyset^{(2)}, \mathcal{R}_\emptyset^{(3)} \}$ using the neighborhood of depth $k = 3$ of the root node $\emptyset$, i.e., $\mathcal{T}^{(3)}(\boldsymbol{\mathcal{X}})$.} \label{F.Construction}
\end{figure}

If the unmarked coupling of $\mathcal{G}_I^{(k)}$ and $\mathcal{T}^{(k)}$ is unsuccessful, then for any node $\mathbf{i} \in \mathcal{T}^{(k)}$ that does not have a coupled vertex in $\mathcal{G}_I^{(k)}$, or that has an offspring $(\mathbf{i},j)$ that does not have a coupled vertex in $\mathcal{G}_I^{(k)}$, sample an independent copy of the noise vector $(\zeta, \{ \xi_j\}_{j \geq 1})$, say $(\zeta_\mathbf{i}, \{ \xi_{(\mathbf{i},j)} \}_{j \geq 1}) $ and let
$$\mathcal{V}_\mathbf{i}^{(r)} =  \Psi\left( {\boldsymbol{\mathcal{X}}_\mathbf{i}}, \zeta_\mathbf{i}, \{ \mathcal{V}_{(\mathbf{i},j)}^{(r-1)}, \xi_{ (\mathbf{i},j)} : 1 \leq j \leq \mathcal{N}_\mathbf{i} \} \right), \qquad \text{if }  \mathbf{i} \in \mathcal{A}_{k-r}, \, 1 \leq r < k,$$
or
$$\mathcal{R}_\emptyset^{(k)} = \Phi\left( {\boldsymbol{\mathcal{X}}_\emptyset}, \zeta_\emptyset , \{ \mathcal{V}_{j}^{(k-1)} , \xi_j: 1 \leq j \leq \mathcal{N}_\emptyset \} \right), \qquad \text{if } \mathbf{i} = \emptyset.$$ 

The first lemma below establishes the finiteness of the $p$th moment of $\mathcal{R}_\emptyset^{(k)}$ and $\mathcal{V}^{(k)}_1$. Throughout the paper, we use the convention $\sum_{i=a}^b x_i \equiv 0$ if $a > b$. 

\begin{lemma} \label{L.Moments}
{Suppose Assumption~\ref{A.PhiMap} holds, the graph sequence $\mathbf{G}$ satisfies Assumption~\ref{A.RGmodels}, and $r_0^p :=  E\left[ \left|R_1^{(0)}\right|^p \right] < \infty$. Then, we have that for any $k \geq 0$, 
\begin{align*}
\left( E\left[ \left|\mathcal{V}_1^{(k)} \right|^p \right] \right)^{1/p} &\leq \left(  \left( E\left[ \left(\sigma_+(\boldsymbol{\mathcal{X}}_1) \beta(\boldsymbol{\mathcal{X}}_1) \right)^p \right] \right)^{1/p} + \left( E\left[ \left|g(0, \boldsymbol{\mathcal{X}}_1) \right|^p \right] \right)^{1/p}  \right) \sum_{r=0}^{k-1} c^r \\
&\hspace{5mm} + c^{k} \left(  \left( E\left[  \sigma_+( \boldsymbol{\mathcal{X}}_1)^p \right]   \right)^{1/p}  r_0 +  \left( E\left[ \left|g(0, \boldsymbol{\mathcal{X}}_1) \right|^p \right] \right)^{1/p} \right) < \infty,
\end{align*}
where $c = \left( E[(\mathcal{N}_1 \sigma_-(\boldsymbol{\mathcal{X}}_1) \sigma_+(\boldsymbol{\mathcal{X}}_1))^p] \right)^{1/p}$ and 
$$\left( E\left[ \left|\mathcal{R}_\emptyset^{(k+1)} \right|^p \right] \right)^{1/p} \leq \left( E\left[ \left| \mathcal{V}_1^{(k)} \right|^p \right] \right)^{1/p}  \left( E\left[ \left( \mathcal{N}_\emptyset \sigma_-(\boldsymbol{\mathcal{X}}_\emptyset)    \right)^p \right] \right)^{1/p} +  \left( E\left[ \beta(\boldsymbol{\mathcal{X}}_\emptyset)^p \right] \right)^{1/p} < \infty.$$}
\end{lemma}

\begin{proof}
Note that under {Assumption~\ref{A.PhiMap}(3)} we have for any $k \geq 1$,
\begin{align*}
\left( E\left[ \left|\mathcal{R}^{(k+1)}_\emptyset \right|^p \right] \right)^{1/p} &\leq  \left( E\left[ \left( \sigma_-({\boldsymbol{\mathcal{X}}_\emptyset}) \sum_{j=1}^{\mathcal{N}_\emptyset}  \left| \mathcal{V}_j^{(k)} \right| + \beta({\boldsymbol{\mathcal{X}}_\emptyset}) \right)^p \right] \right)^{1/p} \\
&\leq  \left( E\left[ \left( \sigma_-({\boldsymbol{\mathcal{X}}_\emptyset}) \sum_{j=1}^{\mathcal{N}_\emptyset} \left( E\left[ \left| \mathcal{V}_j^{(k)} \right|^p \right] \right)^{1/p}  + \beta({\boldsymbol{\mathcal{X}}_\emptyset}) \right)^p \right] \right)^{1/p} \\
&= \left( E\left[ \left( \sigma_-({\boldsymbol{\mathcal{X}}_\emptyset}) \mathcal{N}_\emptyset \left( E\left[ \left| \mathcal{V}_1^{(k)} \right|^p \right] \right)^{1/p}  + \beta({\boldsymbol{\mathcal{X}}_\emptyset}) \right)^p \right] \right)^{1/p}  \\
&\leq  \left( E\left[ \left| \mathcal{V}_1^{(k)} \right|^p \right] \right)^{1/p}  \left( E\left[ \left( \mathcal{N}_\emptyset \sigma_-({\boldsymbol{\mathcal{X}}_\emptyset})    \right)^p \right] \right)^{1/p} +  \left( E\left[ \beta({\boldsymbol{\mathcal{X}}_\emptyset})^p \right] \right)^{1/p},
\end{align*}
where in the second and third inequalities we used Minkowski's inequality, the first time conditionally on {$\boldsymbol{\mathcal{X}}_\emptyset$ and using the independence of the $\{\mathcal{V}_j^{(k)}: j \geq 1\}$ and $\boldsymbol{\mathcal{X}}_\emptyset$}. {The last two expectations are finite by Assumption~\ref{A.RGmodels}.}

{Under Assumption~\ref{A.PhiMap}(2) and a} slight modification of the same steps yields

\begin{align*}
\left( E\left[ \left|\mathcal{V}^{(k)}_1 \right|^p \right] \right)^{1/p} &\leq \left( E\left[ {\sigma_+(\boldsymbol{\mathcal{X}}_1)^p} \left| \mathcal{R}_1^{(k-1)}  \right|^p \right] \right)^{1/p} + \left( E\left[ \left|g(0, {\boldsymbol{\mathcal{X}}_1}) \right|^p \right] \right)^{1/p}  \\
&\leq \left( E\left[ \left| \mathcal{V}_1^{(k-1)} \right|^p \right] \right)^{1/p}  \left( E\left[ \left( \mathcal{N}_1 \sigma_-({\boldsymbol{\mathcal{X}}_1}) \sigma_+({\boldsymbol{\mathcal{X}}_1})   \right)^p \right] \right)^{1/p} +  \left( E\left[ \left(\sigma_+({\boldsymbol{\mathcal{X}}_1}) \beta({\boldsymbol{\mathcal{X}}_1}) \right)^p \right] \right)^{1/p} \\
&\hspace{5mm} + \left( E\left[ \left|g(0, {\boldsymbol{\mathcal{X}}_1}) \right|^p \right] \right)^{1/p} \\
&= c \left( E\left[ \left| \mathcal{V}_1^{(k-1)} \right|^p \right] \right)^{1/p} + d,
\end{align*}
where $d = \left( E\left[ \left(\sigma_+({\boldsymbol{\mathcal{X}}_1}) \beta({\boldsymbol{\mathcal{X}}_1}) \right)^p \right] \right)^{1/p} + \left( E\left[ \left|g(0, {\boldsymbol{\mathcal{X}}_1}) \right|^p \right] \right)^{1/p} $, {and $d < \infty$ by Assumption~\ref{A.RGmodels}}. Now let $a_k = \left( E\left[ \left|\mathcal{V}^{(k)}_1 \right|^p \right] \right)^{1/p}$ and {iterate} the inequality we just derived {to obtain}:
$$a_{k} \leq d + c a_{k-1} \leq d \sum_{r=0}^{k-1} c^r + c^{k} a_0,$$
where $a_0^p = E\left[ \left|g(R_1^{(0)}, {\boldsymbol{\mathcal{X}}_1})\right|^p \right]$, $R_1^{(0)}$ is distributed according to $\mu_0$, and is independent of ${\boldsymbol{\mathcal{X}}_1}$. To complete the proof note that
\begin{align*}
a_0 &\leq \left( E\left[ \left| g(R_1^{(0)}, {\boldsymbol{\mathcal{X}}_1}) -g(0, {\boldsymbol{\mathcal{X}}_1}) \right|^p \right] \right)^{1/p} +  \left( E\left[ \left|g(0, {\boldsymbol{\mathcal{X}}_1}) \right|^p \right] \right)^{1/p} \\
&\leq  \left( E\left[  \sigma_+( {\boldsymbol{\mathcal{X}}_1})^p \right]  E\left[ \left| R_1^{(0)}  \right|^p \right] \right)^{1/p}  +  \left( E\left[ \left|g(0, {\boldsymbol{\mathcal{X}}_1}) \right|^p \right] \right)^{1/p}. 
\end{align*} 
\end{proof}

\begin{theorem} \label{T.CouplingHolds}
Suppose Assumption~\ref{A.PhiMap} holds and {the graph sequence $\mathbf{G} = \{ G(V_n, E_n; \mathscr{A}_n): n \geq 1\}$} satisfies Assumption~\ref{A.RGmodels}. For any $k \geq 0$, $\epsilon \in (0,1)$, and 
$$F_\emptyset^{(k,\epsilon)} =\bigcap_{s=0}^k \bigcap_{\mathbf{i} \in \mathcal{A}_s}  \left\{  \rho(\mathbf{X}_{\theta(\mathbf{i})}, {\boldsymbol{\mathcal{X}}_{\mathbf{i}}}) \leq \epsilon \right\},$$
we have
\begin{align*}
&\left( \mathbb{E}_n\left[ \left| R_I^{(k)} - \mathcal{R}_\emptyset^{(k)} \right|^p 1(F_\emptyset^{(k,\epsilon)})  \right] \right)^{1/p} \\
&\leq  H \epsilon^\alpha \left( 1 + \left( E\left[ (\sigma_-({\boldsymbol{\mathcal{X}}_\emptyset}) \mathcal{N}_\emptyset)^p \right] \right)^{1/p} \left( E\left[ \left| \mathcal{V}_1^{(k-1)} \right|^p \right] \right)^{1/p} + \left( E[\beta({\boldsymbol{\mathcal{X}}_\emptyset})^p ] \right)^{1/p} \right) \\
&\hspace{5mm} + (1+ H\epsilon^\alpha) \left( E\left[ (\sigma_-({\boldsymbol{\mathcal{X}}_\emptyset}) \mathcal{N}_\emptyset)^p \right] \right)^{1/p} b_{k-1},
\end{align*}
where $\{b_r: r \geq 1\}$ is a sequence satisfying $b_0 = Q\epsilon^\gamma \left(1 + r_0 \left(E[ \sigma_+({\boldsymbol{\mathcal{X}}_1})^p ] \right)^{1/p}\right)$, $r_0^p =  E[|R_1^{(0)}|^p] $ and for $1 \leq r < k$, 
\begin{align*}
b_r &\leq w(\epsilon) \left( 1 + \left( E\left[ \sigma_+({\boldsymbol{\mathcal{X}}_1})^p \right] \right)^{1/p} + c \left( E\left[ \left| \mathcal{V}_1^{(r-1)}\right|^p \right] \right)^{1/p} + \left(E \left[ (\sigma_+({\boldsymbol{\mathcal{X}}_1})\beta({\boldsymbol{\mathcal{X}}_1}))^p \right] \right)^{1/p}  \right) \\
&\hspace{5mm} + (1+w(\epsilon)) c b_{r-1},
\end{align*}
with $c = \left( E[ (\mathcal{N}_1 \sigma_-({\boldsymbol{\mathcal{X}}_1}) \sigma_+({\boldsymbol{\mathcal{X}}_1}))^p ]\right)^{1/p}$ and $w(\epsilon) = H \epsilon^\alpha + Q\epsilon^\gamma + HQ \epsilon^{\alpha+\gamma}$. 
\end{theorem}

\begin{proof}
Fix $k \geq 0$, $\epsilon \in (0,1)$, and construct the strong coupling of $\mathcal{G}_I^{(k)}(\mathbf{X})$. Define for each $\mathbf{i} \in \mathcal{A}_r$, $0 \leq r \leq k$, the events:
$$F^{(k-r,\epsilon)}_\mathbf{i} = \bigcap_{s=r}^k \bigcap_{(\mathbf{i}, \mathbf{j}) \in \mathcal{A}_s} \left\{  \rho(\mathbf{X}_{\theta((\mathbf{i},\mathbf{j}))}, {\boldsymbol{\mathcal{X}}_{(\mathbf{i},\mathbf{j})}} ) \leq \epsilon \right\}.$$
Note that on the event $F_\emptyset^{(k,\epsilon)}$, we have $N_{\theta(\mathbf{i})} = \mathcal{N}_\mathbf{i}$ for all $\mathbf{i} \in \mathcal{T}^{(k)}$. Next, define the intermediate random variables for $0 \leq r < k$ and $\mathbf{i} \in \mathcal{A}_r$:
\begin{align*}
\tilde R_{\theta(\mathbf{i})}^{(k-r)} &= \Phi\left( {\boldsymbol{\mathcal{X}}_\mathbf{i}},  \zeta_{\theta(\mathbf{i})}^{(k-r-1)} , \left\{ V_{\theta(\mathbf{i})}^{(k-r-1)} , \xi_{\theta((\mathbf{i},j)), \theta(\mathbf{i})}^{(k-r-1)}: 1 \leq j \leq \mathcal{N}_\mathbf{i} \right\} \right),  
\end{align*}
where $\theta(\emptyset) = I$. 

We now define
$$\mathscr{H}_n = \sigma\left( \mathscr{G}_n \cup \sigma\left( I, \mathcal{T}({\boldsymbol{\mathcal{X}}}) \right) \right),$$
where $\mathcal{T}({\boldsymbol{\mathcal{X}}})$ is the marked Galton-Watson process whose restriction to its first $k$ generations, i.e, $\mathcal{T}^{(k)}({\boldsymbol{\mathcal{X}}})$, is coupled to $\mathcal{G}_I^{(k)}(\mathbf{X})$. Note that the event $F_\emptyset^{(k,\epsilon)}$ is measurable with respect to $\mathscr{H}_n$, and the only remaining randomness is left is the one produced by the noises. From now on, assume that we are on the event $F_\emptyset^{(k,\epsilon)}$, so $\rho(\mathbf{X}_{\theta(\mathbf{i})}, {\boldsymbol{\mathcal{X}}_\mathbf{i}}) \leq \epsilon$ for all $\mathbf{i} \in \mathcal{T}^{(k)}$. 

Next, note that by Minkowski's inequality, 
\begin{align*}
\left( E\left[ \left. \left| R_I^{(k)} - \mathcal{R}_\emptyset^{(k)} \right|^p \right| \mathscr{H}_n  \right] \right)^{1/p} &\leq \left( E\left[ \left. \left| R_I^{(k)} - \tilde R_I^{(k)} \right|^p  \right| \mathscr{H}_n \right] \right)^{1/p} + \left( E\left[  \left. \left| \tilde R_I^{(k)} - \mathcal{R}_\emptyset^{(k)} \right|^p \right| \mathscr{H}_n \right] \right)^{1/p} .
\end{align*}
Now condition on the $\{ V_j^{(k-1)}: j \in V_n\}$, and use Assumption~\ref{A.PhiMap}(5) to obtain that 
\begin{align*}
&\left( E\left[\left. \left| R_I^{(k)} - \tilde R_I^{(k)} \right|^p \right| \mathscr{H}_n \right] \right)^{1/p} \\
&\leq \left( E\left[ \left. \left( H \epsilon^\alpha \left(1 \vee \left( \sum_{j=1}^{\mathcal{N}_\emptyset} \sigma_-({\boldsymbol{\mathcal{X}}_\emptyset}) \left| V_{\theta(j)}^{(k-1)} \right| + \beta({\boldsymbol{\mathcal{X}}_\emptyset}) \right) \right)  \right)^p \right| \mathscr{H}_n  \right] \right)^{1/p} \\
&\leq H \epsilon^\alpha \left( 1 +  \left(  E\left[ \left. \left(  \sum_{j=1}^{\mathcal{N}_\emptyset} \sigma_-({\boldsymbol{\mathcal{X}}_\emptyset}) \left| V_{\theta(j)}^{(k-1)} \right|   \right)^p  \right| \mathscr{H}_n \right] \right)^{1/p} +  \beta({\boldsymbol{\mathcal{X}}_\emptyset})  \right) \\
&\leq H \epsilon^\alpha \left( 1 + \sigma_-({\boldsymbol{\mathcal{X}}_\emptyset}) \sum_{j=1}^{\mathcal{N}_\emptyset}  \left(  E \left[ \left. \left| V_{\theta(j)}^{(k-1)} \right|^p \right| \mathscr{H}_n  \right] \right)^{1/p} +  \beta({\boldsymbol{\mathcal{X}}_\emptyset}) \right) \\
&\leq H \epsilon^\alpha \left( 1 + \sigma_-({\boldsymbol{\mathcal{X}}_\emptyset}) \sum_{j=1}^{\mathcal{N}_\emptyset}  \left( \Delta_j^{(k-1)} + \mathcal{Y}_j^{(k-1)}  \right) +  \beta({\boldsymbol{\mathcal{X}}_\emptyset}) \right)
\end{align*}
where for $\mathbf{i} \in \mathcal{A}_{r}$, $1 \leq r \leq k$,
$$\Delta_\mathbf{i}^{(k-r)} =  \left(  E \left[ \left. \left| V_{\theta(\mathbf{i})}^{(k-r)} - \mathcal{V}_\mathbf{i}^{(k-r)} \right|^p \right| \mathscr{H}_n  \right] \right)^{1/p} \qquad \text{and} \qquad \mathcal{Y}_\mathbf{i}^{(k-r)} =  \left(  E \left[ \left. \left|  \mathcal{V}_\mathbf{i}^{(k-r)} \right|^p \right| \mathcal{T}({\boldsymbol{\mathcal{X}}})  \right] \right)^{1/p}.$$
{Assumption~\ref{A.PhiMap}(1)} also gives
\begin{align*}
\left( E\left[\left. \left| \tilde R_I^{(k)}- \mathcal{R}_\emptyset^{(k)}  \right|^p \right| \mathscr{H}_n \right] \right)^{1/p} 
&\leq \left( E\left[\left. \left( \sigma_-({\boldsymbol{\mathcal{X}}_\emptyset}) \sum_{j=1}^{\mathcal{N}_\emptyset} \left| V_{\theta(j)}^{(k-1)} - \mathcal{V}_j^{(k-1)} \right|  \right)^p \right| \mathscr{H}_n \right] \right)^{1/p} \\
&\leq \sigma_-({\boldsymbol{\mathcal{X}}_\emptyset}) \sum_{j=1}^{\mathcal{N}_\emptyset} \Delta_j^{(k-1)} .
\end{align*}
Hence,
\begin{align*}
&\left( E\left[ \left. \left| R_I^{(k)} - \mathcal{R}_\emptyset^{(k)} \right|^p \right| \mathscr{H}_n  \right] \right)^{1/p} \\
&\leq H \epsilon^\alpha \left( 1 + \sigma_-({\boldsymbol{\mathcal{X}}_\emptyset}) \sum_{j=1}^{\mathcal{N}_\emptyset}   \mathcal{Y}_j^{(k-1)}   +  \beta({\boldsymbol{\mathcal{X}}_\emptyset}) \right) + (1 + H\epsilon^\alpha) \sigma_-({\boldsymbol{\mathcal{X}}_\emptyset}) \sum_{j=1}^{\mathcal{N}_\emptyset} \Delta_j^{(k-1)}. 
\end{align*}

We will now derive an upper bound for $\Delta_\mathbf{i}^{(k-r)}$. By {Assumption~\ref{A.PhiMap}, items (2) and (5)} , we have that for any $1 \leq r < k$, and $\mathbf{i} \in \mathcal{A}_{r}$:
\begin{align*}
\Delta_\mathbf{i}^{(k-r)} &\leq \left(  E \left[ \left. \left| g( R_{\theta(\mathbf{i})}^{(k-r)}, \mathbf{X}_{\theta(\mathbf{i})}) - g( R_{\theta(\mathbf{i})}^{(k-r)}, {\boldsymbol{\mathcal{X}}_\mathbf{i}}) \right|^p \right| \mathscr{H}_n  \right] \right)^{1/p} \\
&\hspace{5mm} + \left(  E \left[ \left. \left| g( R_{\theta(\mathbf{i})}^{(k-r)}, {\boldsymbol{\mathcal{X}}_\mathbf{i}}) - g(\mathcal{R}_\mathbf{i}^{(k-r)}, {\boldsymbol{\mathcal{X}}_\mathbf{i}}) \right|^p \right| \mathscr{H}_n  \right] \right)^{1/p} \\
&\leq  \left(  E \left[ \left. \left( Q \epsilon^{\gamma} \left(1 \vee \sigma_+( {\boldsymbol{\mathcal{X}}_\mathbf{i}}) | R_{\theta(\mathbf{i})}^{(k-r)}|  \right)\right)^p \right| \mathscr{H}_n  \right] \right)^{1/p} \\
&\hspace{5mm} + \left(  E \left[ \left. \sigma_+( {\boldsymbol{\mathcal{X}}_\mathbf{i}})^p \left| R_{\theta(\mathbf{i})}^{(k-r)} - \mathcal{R}_\mathbf{i}^{(k-r)} \right|^p \right| \mathscr{H}_n  \right] \right)^{1/p} \\
&\leq Q\epsilon^\gamma \left( 1 + \sigma_+({\boldsymbol{\mathcal{X}}_\mathbf{i}}) \left( E\left[ \left. \left| R_{\theta(\mathbf{i})}^{(k-r)} \right|^p \right| \mathscr{H}_n \right] \right)^{1/p} \right) + \sigma_+( {\boldsymbol{\mathcal{X}}_\mathbf{i}}) \left(  E \left[ \left.  \left| R_{\theta(\mathbf{i})}^{(k-r)} - \mathcal{R}_\mathbf{i}^{(k-r)} \right|^p \right| \mathscr{H}_n  \right] \right)^{1/p} \\
&\leq Q\epsilon^\gamma \left( 1 + \sigma_+({\boldsymbol{\mathcal{X}}_\mathbf{i}}) \left( E\left[ \left. \left| \mathcal{R}_\mathbf{i}^{(k-r)} \right|^p \right| \mathcal{T}({\boldsymbol{\mathcal{X}}}) \right] \right)^{1/p} \right) \\
&\hspace{5mm} + (1+ Q\epsilon^\gamma) \sigma_+( {\boldsymbol{\mathcal{X}}_\mathbf{i}}) \left(  E \left[ \left.  \left| R_{\theta(\mathbf{i})}^{(k-r)} - \mathcal{R}_\mathbf{i}^{(k-r)} \right|^p \right| \mathscr{H}_n  \right] \right)^{1/p}.
\end{align*}
Moreover,
\begin{align*}
\left( E\left[ \left. \left| \mathcal{R}_\mathbf{i}^{(k-r)} \right|^p \right| \mathcal{T}({\boldsymbol{\mathcal{X}}}) \right] \right)^{1/p} &\leq \sigma_-({\boldsymbol{\mathcal{X}}_\mathbf{i}}) \sum_{j=1}^{\mathcal{N}_\mathbf{i}} \mathcal{Y}_{(\mathbf{i},j)}^{(k-r-1)} + \beta({\boldsymbol{\mathcal{X}}_\mathbf{i}}),
\end{align*}
and by the same arguments used above,
\begin{align*}
& \left(  E \left[ \left.  \left| R_{\theta(\mathbf{i})}^{(k-r)} - \mathcal{R}_\mathbf{i}^{(k-r)} \right|^p \right| \mathscr{H}_n  \right] \right)^{1/p} \\
&\leq H \epsilon^\alpha \left( 1 + \sigma_-({\boldsymbol{\mathcal{X}}_\mathbf{i}}) \sum_{j=1}^{\mathcal{N}_\mathbf{i}}   \mathcal{Y}_{(\mathbf{i},j)}^{(k-r-1)}   +  \beta({\boldsymbol{\mathcal{X}}_\mathbf{i}}) \right) + (1 + H\epsilon^\alpha) \sigma_-({\boldsymbol{\mathcal{X}}_\mathbf{i}}) \sum_{j=1}^{\mathcal{N}_\mathbf{i}} \Delta_{(\mathbf{i},j)}^{(k-r-1)}.
\end{align*} 
Hence,
\begin{align*}
\Delta_\mathbf{i}^{(k-r)} &\leq Q\epsilon^\gamma \left( 1 + \sigma_+({\boldsymbol{\mathcal{X}}_\mathbf{i}}) \sigma_-({\boldsymbol{\mathcal{X}}_\mathbf{i}}) \sum_{j=1}^{\mathcal{N}_\mathbf{i}} \mathcal{Y}_{(\mathbf{i},j)}^{(k-r-1)} + \sigma_+({\boldsymbol{\mathcal{X}}_\mathbf{i}}) \beta({\boldsymbol{\mathcal{X}}_\mathbf{i}}) \right) \\
&\hspace{5mm} + (1+ Q\epsilon^\gamma) H \epsilon^\alpha   \left( \sigma_+( {\boldsymbol{\mathcal{X}}_\mathbf{i}}) + \sigma_+( {\boldsymbol{\mathcal{X}}_\mathbf{i}}) \sigma_-({\boldsymbol{\mathcal{X}}_\mathbf{i}}) \sum_{j=1}^{\mathcal{N}_\mathbf{i}}   \mathcal{Y}_{(\mathbf{i},j)}^{(k-r-1)}   + \sigma_+( {\boldsymbol{\mathcal{X}}_\mathbf{i}}) \beta({\boldsymbol{\mathcal{X}}_\mathbf{i}}) \right)  \\
&\hspace{5mm} + (1+ Q\epsilon^\gamma) (1 + H\epsilon^\alpha)  \sigma_+( {\boldsymbol{\mathcal{X}}_\mathbf{i}}) \sigma_-({\boldsymbol{\mathcal{X}}_\mathbf{i}}) \sum_{j=1}^{\mathcal{N}_\mathbf{i}} \Delta_{(\mathbf{i},j)}^{(k-r-1)} \\
&\leq w(\epsilon) \mathcal{W}_\mathbf{i}^{(k-r)} + (1 + w(\epsilon))  \sigma_+( {\boldsymbol{\mathcal{X}}_\mathbf{i}}) \sigma_-({\boldsymbol{\mathcal{X}}_\mathbf{i}}) \sum_{j=1}^{\mathcal{N}_\mathbf{i}} \Delta_{(\mathbf{i},j)}^{(k-r-1)},
\end{align*}
where $w(\epsilon) = Q\epsilon^\gamma + H\epsilon^\alpha + QH \epsilon^{\alpha+\gamma}$ and
$$\mathcal{W}_\mathbf{i}^{(k-r)} = 1+ \sigma_+( {\boldsymbol{\mathcal{X}}_\mathbf{i}}) + \sigma_+( {\boldsymbol{\mathcal{X}}_\mathbf{i}}) \sigma_-({\boldsymbol{\mathcal{X}}_\mathbf{i}}) \sum_{j=1}^{\mathcal{N}_\mathbf{i}}   \mathcal{Y}_{(\mathbf{i},j)}^{(k-r-1)}   + \sigma_+({ \boldsymbol{\mathcal{X}}_\mathbf{i}}) \beta({\boldsymbol{\mathcal{X}}_\mathbf{i}}).$$

Note that for $\mathbf{i} \in \mathcal{A}_k$ we have $R_{\theta(\mathbf{i})}^{(0)} = \mathcal{R}_\mathbf{i}^{(0)}$, and therefore,
$$\Delta_\mathbf{i}^{(0)} = \left( E\left[\left. \left| g(R_{\theta(\mathbf{i})}^{(0)}, \mathbf{X}_{\theta(\mathbf{i})}) - g(\mathcal{R}_\mathbf{i}^{(0)}, {\boldsymbol{\mathcal{X}}_{\mathbf{i}}}) \right|^p  \right| \mathscr{H}_n \right] \right)^{1/p} \leq Q\epsilon^\gamma \left( 1 + \sigma_+({\boldsymbol{\mathcal{X}}_\mathbf{i}}) r_0 \right),$$
where $r_0 = \left( E[ |R_1^{(0)}|^p ] \right)^{1/p}$. 

Iterating the recursion for $\Delta_\mathbf{i}^{(k-r)}$ gives for $j \in \mathcal{A}_1$:
\begin{align*}
\Delta_j^{(k-1)} &\leq \mathcal{L}_j^{(k-1)},
\end{align*}
where $\mathcal{L}_\mathbf{i}^{(k-r)}$ for $\mathbf{i} \in \mathcal{A}_r$, $1 \leq r \leq k$ satisfies
\begin{align*}
\mathcal{L}_\mathbf{i}^{(k-r)} &= w(\epsilon) \mathcal{W}_\mathbf{i}^{(k-r)}  + (1+w(\epsilon))  \sigma_+( {\boldsymbol{\mathcal{X}}_\mathbf{i}}) \sigma_-({\boldsymbol{\mathcal{X}}_\mathbf{i}}) \sum_{j=1}^{\mathcal{N}_\mathbf{i}} \mathcal{L}_{(\mathbf{i},j)}^{(k-r-1)}, \qquad \mathcal{L}_\mathbf{i}^{(0)} = Q\epsilon^\gamma \left( 1 + \sigma_+({\boldsymbol{\mathcal{X}}_\mathbf{i}}) r_0 \right). 
\end{align*}

It only remains to compute the expectation using Minkowski's inequality and the branching property on $\mathcal{T}({\boldsymbol{\mathcal{X}}})$, which we do as follows:
\begin{align*}
&\left( \mathbb{E}_n\left[ \left| R_I^{(k)} - \mathcal{R}_\emptyset^{(k)} \right|^p 1(F_\emptyset^{(k,\epsilon)}) \right] \right)^{1/p} \\
&\leq \left( \mathbb{E}_n\left[ E\left[ \left. \left| R_I^{(k)} - \mathcal{R}_\emptyset^{(k)} \right|^p \right| \mathscr{H}_n \right]  1(F_\emptyset^{(k,\epsilon)}) \right] \right)^{1/p} \\
&\leq \left( E\left[ \left( H \epsilon^\alpha \left( 1 + \sigma_-({\boldsymbol{\mathcal{X}}_\emptyset}) \sum_{j=1}^{\mathcal{N}_\emptyset}   \mathcal{Y}_j^{(k-1)}   +  \beta({\boldsymbol{\mathcal{X}}_\emptyset}) \right) + (1 + H\epsilon^\alpha) \sigma_-({\boldsymbol{\mathcal{X}}_\emptyset}) \sum_{j=1}^{\mathcal{N}_\emptyset} \mathcal{L}_j^{(k-1)}  \right)^p  \right] \right)^{1/p} \\
&\leq H \epsilon^\alpha \left( 1 + \left( E\left[ (\sigma_-({\boldsymbol{\mathcal{X}}_\emptyset}) \mathcal{N}_\emptyset)^p \right] \right)^{1/p} \left( E\left[ (\mathcal{Y}_1^{(k-1)})^p \right] \right)^{1/p} + \left( E[\beta({\boldsymbol{\mathcal{X}}_\emptyset})^p ] \right)^{1/p} \right) \\
&\hspace{5mm} + (1+ H\epsilon^\alpha) \left( E\left[ (\sigma_-({\boldsymbol{\mathcal{X}}_\emptyset}) \mathcal{N}_\emptyset)^p \right] \right)^{1/p} \left( E\left[ (\mathcal{L}_1^{(k-1)})^p \right] \right)^{1/p},
\end{align*}
where $E\left[ (\mathcal{Y}_1^{(k-1)})^p \right]  = E\left[ \left| \mathcal{V}_1^{(k-1)} \right|^p \right] < \infty$ by Lemma~\ref{L.Moments}. Let $b_{k-r} = \left( E\left[ (\mathcal{L}_\mathbf{i}^{(k-r)})^p \right] \right)^{1/p}$ for $\mathbf{i} \in \mathcal{A}_r$, $1 \leq r \leq k$. Then, Minkowski's inequality and the branching property give again:
\begin{align*}
b_{k-r} &\leq w(\epsilon) \left( E\left[( \mathcal{W}_\mathbf{i}^{(k-r)})^p \right] \right)^{1/p} +(1+w(\epsilon)) \left( E\left[ (\sigma_+({\boldsymbol{\mathcal{X}}_1})) \sigma_-({\boldsymbol{\mathcal{X}}_1}) \mathcal{N}_1)^p \right] \right)^{1/p} b_{k-r-1} \\
&\leq w(\epsilon) \left( 1 + \left( E\left[ \sigma_+({\boldsymbol{\mathcal{X}}_1})^p \right] \right)^{1/p} + c \left( E\left[ (\mathcal{Y}_{(\mathbf{i},1)}^{(k-r-1)})^p \right] \right)^{1/p} + \left(E \left[ (\sigma_+({\boldsymbol{\mathcal{X}}_1})\beta({\boldsymbol{\mathcal{X}}_1}))^p \right] \right)^{1/p}  \right) \\
&\hspace{5mm} + (1+w(\epsilon)) c b_{k-r-1}
\end{align*}
where $c = \left( E\left[ (\sigma_+({\boldsymbol{\mathcal{X}}_1})) \sigma_-({\boldsymbol{\mathcal{X}}_1}) \mathcal{N}_1)^p \right] \right)^{1/p} $. This completes the proof.
\end{proof}

The following result provides a more explicit bound for the expectation in Theorem~\ref{T.CouplingHolds}. Note that the bound can be made uniform in $k$ when $c < 1$. 

\begin{cor} \label{C.Explicit}
Under the same assumptions as Theorem~\ref{T.CouplingHolds}, we have that
\begin{align*}
\left( \mathbb{E}_n\left[ \left| R_I^{(k)} - \mathcal{R}_\emptyset^{(k)} \right|^p 1(F_\emptyset^{(k,\epsilon)})  \right] \right)^{1/p} \leq  H_{k,\epsilon} w(\epsilon) ,
\end{align*}
where $w(\epsilon) = Q\epsilon^\gamma + H\epsilon^\alpha + QH \epsilon^{\alpha+\gamma}$,
\begin{align*}
H_{k,\epsilon} &= B \left( 1 + (1 + w(\epsilon))    \sum_{i=0}^{k-1} \sum_{j=0}^{k-1-i} c^j (1+ w(\epsilon))^i c^i \right) \qquad \text{and} \\
B &= 1  +  \left( E[\beta({\boldsymbol{\mathcal{X}}_\emptyset})^p ] \right)^{1/p} + 2\left( E\left[ (\sigma_-({\boldsymbol{\mathcal{X}}_\emptyset}) \mathcal{N}_\emptyset)^p \right] \right)^{1/p} \\
&\hspace{5mm} \times \left( 1 + \left( E\left[ \sigma_+({\boldsymbol{\mathcal{X}}_1})^p \right] \right)^{1/p}  +   \left( E\left[ \left(\sigma_+({\boldsymbol{\mathcal{X}}_1})( \beta({\boldsymbol{\mathcal{X}}_1}) \vee r_0) \right)^p \right] \right)^{1/p} + \left( E\left[ \left|g(0, {\boldsymbol{\mathcal{X}}_1}) \right|^p \right] \right)^{1/p} \right).
\end{align*}
\end{cor}

\begin{proof}
Define $q = \left( E\left[ \left(\sigma_+({\boldsymbol{\mathcal{X}}_1})( \beta({\boldsymbol{\mathcal{X}}_1}) \vee r_0) \right)^p \right] \right)^{1/p} + \left( E\left[ \left|g(0, {\boldsymbol{\mathcal{X}}_1}) \right|^p \right] \right)^{1/p}$ and note that by Lemma~\ref{L.Moments} we have for any $r \geq 0$,
$$\left( E\left[ \left| \mathcal{V}_1^{(r)} \right|^p \right] \right)^{1/p} \leq q \sum_{j=0}^r c^j.$$
Now use Theorem~\ref{T.CouplingHolds} to obtain that
\begin{align*}
b_r &\leq w(\epsilon) \left( 1 + \left( E\left[ \sigma_+({\boldsymbol{\mathcal{X}}_1})^p \right] \right)^{1/p} + \left(E \left[ (\sigma_+({\boldsymbol{\mathcal{X}}_1})\beta({\boldsymbol{\mathcal{X}}_1}))^p \right] \right)^{1/p} + c q \sum_{j=0}^{r-1} c^j   \right) \\
&\hspace{5mm} + (1+w(\epsilon)) c b_{r-1} \\
&\leq w(\epsilon) \tilde q \sum_{j=0}^r c^j + (1+w(\epsilon))c b_{r-1},
\end{align*}
for $\tilde q = q + 1 + \left( E\left[ \sigma_+({\boldsymbol{\mathcal{X}}_1})^p \right] \right)^{1/p} $. Iterating the recursion gives:
\begin{align*}
b_{r} &\leq w(\epsilon) \tilde q \sum_{j=0}^{r} c^j + (1+w(\epsilon))c \left( w(\epsilon) \tilde q \sum_{j=0}^{r-1} c^j + (1+w(\epsilon))c b_{r-2} \right) \\
&\leq \sum_{i=0}^{r-1} w(\epsilon) \tilde q \sum_{j=0}^{r-i} c^j (1+w(\epsilon))^i c^i + (1+w(\epsilon))^{r} c^{r} b_0 \\
&\leq \sum_{i=0}^{r-1} w(\epsilon) \tilde q \sum_{j=0}^{r-i} c^j (1+w(\epsilon))^i c^i + (1+w(\epsilon))^{r} c^{r} w(\epsilon) \tilde q \\
&= w(\epsilon) \tilde q  \sum_{i=0}^{r} \sum_{j=0}^{r-i} c^j (1+ w(\epsilon))^i c^i .
\end{align*}
To complete the proof, use Theorem~\ref{T.CouplingHolds} and Lemma~\ref{L.Moments} again to obtain that
\begin{align*}
&\left( \mathbb{E}_n\left[ \left| R_I^{(k)} - \mathcal{R}_\emptyset^{(k)} \right|^p 1(F_\emptyset^{(k,\epsilon)})  \right] \right)^{1/p} \\
&\leq  H \epsilon^\alpha \left( 1 + \left( E\left[ (\sigma_-({\boldsymbol{\mathcal{X}}_\emptyset}) \mathcal{N}_\emptyset)^p \right] \right)^{1/p} q \sum_{j=0}^{k-1} c^j + \left( E[\beta({\boldsymbol{\mathcal{X}}_\emptyset})^p ] \right)^{1/p} \right) \\
&\hspace{5mm} + (1+ H\epsilon^\alpha) \left( E\left[ (\sigma_-({\boldsymbol{\mathcal{X}}_\emptyset}) \mathcal{N}_\emptyset)^p \right] \right)^{1/p} b_{k-1} \\
&\leq w(\epsilon) \left( 1 + \left( E[\beta({\boldsymbol{\mathcal{X}}_\emptyset})^p ] \right)^{1/p}+  \left( E\left[ (\sigma_-({\boldsymbol{\mathcal{X}}_\emptyset}) \mathcal{N}_\emptyset)^p \right] \right)^{1/p} q \sum_{j=0}^{k-1} c^j  \right) \\
&\hspace{5mm} + (1 + w(\epsilon)) \left( E\left[ (\sigma_-({\boldsymbol{\mathcal{X}}_\emptyset}) \mathcal{N}_\emptyset)^p \right] \right)^{1/p} w(\epsilon) \tilde q  \sum_{i=0}^{k-1} \sum_{j=0}^{k-1-i} c^j (1+ w(\epsilon))^i c^i \\
&\leq  w(\epsilon) \left( 1 + \left( E[\beta({\boldsymbol{\mathcal{X}}_\emptyset})^p ] \right)^{1/p}+ 2 \left( E\left[ (\sigma_-({\boldsymbol{\mathcal{X}}_\emptyset}) \mathcal{N}_\emptyset)^p \right] \right)^{1/p} \tilde q (1 + w(\epsilon))    \sum_{i=0}^{k-1} \sum_{j=0}^{k-1-i} c^j (1+ w(\epsilon))^i c^i  \right) \\
&\leq B w(\epsilon) \left( 1 + (1 + w(\epsilon))    \sum_{i=0}^{k-1} \sum_{j=0}^{k-1-i} c^j (1+ w(\epsilon))^i c^i \right) ,
\end{align*}
with $B = 1  +  \left( E[\beta({\boldsymbol{\mathcal{X}}_\emptyset})^p ] \right)^{1/p} + 2\left( E\left[ (\sigma_-({\boldsymbol{\mathcal{X}}_\emptyset}) \mathcal{N}_\emptyset)^p \right] \right)^{1/p} \tilde q$. 
\end{proof}

Before we move on to the proof of the first part of Theorem~\ref{T.Main}, we first give two technical lemmas.

\begin{lemma} \label{L.CondDp}
Suppose $Z$ and $Z^{(n)} \geq 0$ are such that $d_p(\gamma_n,\gamma)\xrightarrow{P} 0$ where $\gamma_n(\cdot) = \mathbb{P}_n(Z^{(n)}\in\cdot\,)$ and $\gamma (\cdot) = P(Z\in\cdot\,)$ as $n \to \infty$. Then, for any event $\mathcal{E}_n$ constructed on the same probability space as $Z^{(n)}$ and such that $\mathbb{P}_n( \mathcal{E}_n) \xrightarrow{P} 0$ as $n \to \infty$, we have 
$$\mathbb{E}_n\left[ Z^{(n)} 1(\mathcal{E}_n) \right] \xrightarrow{P} 0, \qquad n \to \infty.$$
\end{lemma}

\begin{proof}
To start, fix $M>0$ and note that
\begin{align*}
    \mathbb{E}_n\left[ Z^{(n)} 1(\mathcal{E}_n)\right] &\leq \mathbb{E}_n\left[Z^{(n)}1(\mathcal{E}_n)1(Z^{(n)} \leq M)\right] + \mathbb{E}_n\left[Z^{(n)}1(Z^{(n)} >M))\right]\\
    & \leq M\mathbb{P}_n(\mathcal{E}_n) + \mathbb{E}_n\left[Z^{(n)}1(Z^{(n)}>M) \right]. 
\end{align*}
Since $\mathbb{P}_n( \mathcal{E}_n) \xrightarrow{P} 0$ as $n \to \infty$, we need to show that $\mathbb{E}_n\left[Z^{(n)} 1(Z^{(n)}>M)\right]$ can be made arbitrarily small. To this end, note that
\begin{align*}
    \mathbb{E}_n\left[ Z^{(n)} 1(Z^{(n)} >M)\right] & = \mathbb{E}_n\left[Z^{(n)} \right] - \mathbb{E}_n\left[ Z^{(n)} 1(Z^{(n)}\leq M )\right] ,
\end{align*}
where by Lemma A.2 in \cite{Boll_Jan_Rio_07} and the observation that $z1(z \leq M)$ is bounded and continuous a.e., we obtain that 
\begin{align*}
    \mathbb{E}_n\left[Z^{(n)}1(Z^{(n)} \leq M)\right] \xrightarrow{P}E\left[Z 1(Z \leq M)\right], \qquad n \to \infty
\end{align*}
Therefore, we have that provided $M$ is a point of continuity, 
\begin{align*}
    \mathbb{E}_n\left[Z^{(n)} 1(Z^{(n)}>M)\right] \xrightarrow{P} E\left[Z \right] - E\left[Z1(Z \leq M)\right] = E\left[ Z 1(Z > M) \right] 
\end{align*}
as $n \to \infty$. Since $M$ is arbitrary, take $M \to \infty$ to complete the proof.
\end{proof}

\begin{lemma} \label{L.DecoupledEvents}
Suppose $\Phi$ satisfies Assumption~\ref{A.PhiMap} with part (4)(i) and the graph sequence $\boldsymbol{G}$ satisfies Assumption~\ref{A.RGmodels}. Fix $\epsilon > 0$ and define $J_I^{(k)} = 1 - 1(F_\emptyset^{(k,\epsilon)})$, where the event $F_\emptyset^{(k,\epsilon)}$ is defined in Theorem~\ref{T.CouplingHolds}. 
Then, as $n \to \infty$,
$$\left(  \mathbb{E}_n\left[ \left(  J_I^{(k)} \sigma_-(\mathbf{X}_I) \sum_{j \to I} \sigma_+(\mathbf{X}_j)   \right)^p  \right] \right)^{1/p} +   \left(  \mathbb{E}_n\left[ \left(  J_I^{(k)} \sigma_-(\mathbf{X}_I) \sum_{j \to I} |g(0,\mathbf{X}_j)|   \right)^p  \right] \right)^{1/p} \xrightarrow{P} 0.$$
\end{lemma}

\begin{proof}
Fix $\delta > 0$ and choose $M = M(\delta) > 0$ such that $P(\sigma_+(\boldsymbol{\mathcal{X}}_\emptyset) > M) \leq \delta$. Now note that if we let $\mathbf{y}_M$ denote the vector whose $i$th component is $1(\sigma_+(\mathbf{X}_i) > M)$, then
\begin{align*}
&\left(  \mathbb{E}_n\left[ \left(  J_I^{(k)} \sigma_-(\mathbf{X}_I) \sum_{j \to I} \sigma_+(\mathbf{X}_j)   \right)^p  \right] \right)^{1/p} \\
&\leq \left(  \mathbb{E}_n\left[ \left(   \sigma_-(\mathbf{X}_I) \sum_{j \to I} \sigma_+(\mathbf{X}_j) 1\left( \sigma_+(\mathbf{X}_j) > M \right)  \right)^p  \right] \right)^{1/p} + M  \left(  \mathbb{E}_n\left[ \left(  J_I^{(k)} \sigma_-(\mathbf{X}_I) D_I^-  \right)^p  \right] \right)^{1/p} \\
&= \left( \frac{1}{n} \mathbb{E}_n\left[ \left\| C \, \mathbf{y}_M \right\|_p^p \right] \right)^{1/p} + M  \left(  \mathbb{E}_n\left[ \left(  J_I^{(k)} f_*(\mathbf{X}_I)   \right)^p  \right] \right)^{1/p} . 
\end{align*}
Note that by Assumption~\ref{A.RGmodels} and Lemma~\ref{L.CondDp} we have that
$$ \mathbb{E}_n\left[ \left(  J_I^{(k)} f_*(\mathbf{X}_I)   \right)^p  \right] \xrightarrow{P} 0, \qquad n \to \infty.$$
And by Assumption~\ref{A.PhiMap}(4)(i) and the observation that $\mathbb{E}_n[ \rho(\mathbf{X}_I, {\boldsymbol{\mathcal{X}}_\emptyset})] \xrightarrow{P} 0$ as $n \to \infty$, we have
\begin{align*}
\frac{1}{n} \mathbb{E}_n\left[ \left\| C \, \mathbf{y}_M \right\|_p^p \right] &\leq \frac{1}{n} \mathbb{E}_n\left[ \left\| C \right\|_p^p \left\| \mathbf{y}_M \right\|_p^p \right] \leq K^p \mathbb{P}_n\left( \sigma_+(\mathbf{X}_I) > M \right) \\
&\xrightarrow{P} K^p P(\sigma_+({\boldsymbol{\mathcal{X}}_\emptyset}) > M) < K^p \delta.
\end{align*}
Letting $\delta \downarrow 0$ gives the result. The proof for the expectation involving $|g(0,\mathbf{X}_j)|$ is essentially the same, and is therefore omitted. 
\end{proof}

We are now ready to prove {Part (A) of} Theorem~\ref{T.Main}.

\begin{theorem} \label{T.MainPart1}
Suppose the map $\Phi$ satisfies Assumption~\ref{A.PhiMap}, the directed graph sequence $\boldsymbol{G} = \{ {G(V_n,E_n, \mathscr{A}_n)}: n \geq 1\}$ satisfies Assumption~\ref{A.RGmodels}. Let $\mu_0(\cdot) = P(R_1^{(0)} \in \cdot)$ satisfy $E\left[ |R^{(0)}_1|^p \right] < \infty$ if Assumption~\ref{A.PhiMap}(4)(i) holds, or $\text{supp}(\mu_0) \subseteq [-K, K]$ if Assumption~\ref{A.PhiMap}(4)(ii) holds. Let $\mu_{k,n}(\cdot) = \mathbb{P}_n\left( R_I^{(k)} \in \cdot \right)$, where $I$ is uniformly distributed in $\{1, 2, \dots, n\}$,  independent of $\mathscr{G}_n$. 
Then, for any fixed $k \geq 0$ there exists a sequence of random variables $\{ \mathcal{R}_\emptyset^{(0)}, \mathcal{R}_\emptyset^{(1)}, \dots, \mathcal{R}_\emptyset^{(k)} \}$ {constructed on the coupled marked Galton-Watson tree $\mathcal{T}_{\emptyset(I)}^{(k)}(\boldsymbol{\mathcal{X}})$ from Assumption~\ref{A.RGmodels}, on the same probability space as} $\{ R_I^{(0)}, R_I^{(1)}, \dots, R_I^{(k)} \}$, such that
$$\max_{0 \leq r \leq k} \mathbb{E}_n\left[ \left| R_I^{(r)} - \mathcal{R}_\emptyset^{(r)} \right|^p \right] \xrightarrow{P} 0, \qquad n \to \infty.$$
\end{theorem}

\begin{proof}
For a given $I$, the random variables $\{ \mathcal{R}_\emptyset^{(0)}, \mathcal{R}_\emptyset^{(1)}, \dots, \mathcal{R}_\emptyset^{(k)} \}$ are the ones constructed at the beginning of Section~\ref{S.Proofs}, Note that since $k$ is fixed, it suffices to show that 
$$ \mathbb{E}_n\left[ \left| R_I^{(k)} - \mathcal{R}_\emptyset^{(k)} \right|^p \right] \xrightarrow{P} 0, \qquad n \to \infty.$$

Next, for any $j \in V_n$, define $\emptyset(j)$ to be the root of the coupled tree when $I = j$, and note that we can construct all $n$ couplings of $\mathcal{G}_j^{(k)}$ simultaneously. To simplify the notation, let $J_j^{(r)} = 1 - 1(F_{\emptyset(j)}^{(r,\epsilon)})$, $r \geq 0$, where the event $F_{\emptyset}^{(r,\epsilon)}$ is defined in Theorem~\ref{T.CouplingHolds}, and note that all the $\{ J_i^{(r)} : i \in V_n\}$ are measurable with respect to $\mathscr{G}_n$. Let $\mathbb{I}$ be the identity matrix in $\mathbb{R}^n$, let $\mathcal{J}^{(r)} = \text{diag}( \mathbf{J}^{(r)})$ be the diagonal matrix associated to the vector $\mathbf{J}^{(r)} = (J_1^{(r)}, \dots, J_n^{(r)})'$, and let $\boldsymbol{\mathcal{R}}^{(r)} = (\mathcal{R}_{\emptyset(1)}^{(r)}, \dots, \mathcal{R}_{\emptyset(n)}^{(r)})'$. Then, 
\begin{align*}
\left( \mathbb{E}_n\left[ \left| R_I^{(k)} - \mathcal{R}^{(k)}_\emptyset \right|^p  \right] \right)^{1/p} &\leq \left( \frac{1}{n} \mathbb{E}_n\left[ \left\| (\mathbb{I} - \mathcal{J}^{(k)}) (\mathbf{R}^{(k)} - \boldsymbol{\mathcal{R}}^{(k)}) \right\|_p^p  \right] \right)^{1/p} \\
&\hspace{5mm} + \left( \frac{1}{n} \mathbb{E}_n\left[ \left\|  \mathcal{J}^{(k)} (\mathbf{R}^{(k)}  - \boldsymbol{\mathcal{R}}^{(k)}) \right\|_p^p  \right] \right)^{1/p} .
\end{align*}
Note that if Assumption~\ref{A.PhiMap}(4)(ii) holds, then 
\begin{align*}
\left( \mathbb{E}_n\left[ \left| R_I^{(k)} - \mathcal{R}^{(k)}_\emptyset \right|^p  \right] \right)^{1/p} &\leq \left(  \mathbb{E}_n\left[ \left|  R_I^{(k)} - \mathcal{R}_\emptyset^{(k)} \right|^p 1(F_\emptyset^{(k,\epsilon)})  \right] \right)^{1/p} + 2K \left( \mathbb{E}_n\left[   J_I^{(k)}   \right] \right)^{1/p},
\end{align*}
so Corollary~\ref{C.Explicit} and Assumption~\ref{A.RGmodels} give the result. Therefore, from now on, assume that Assumption~\ref{A.PhiMap}(4)(i) holds, i.e., $\| C\|_p \leq K$ and $\| C^{(0)} \|_p \leq K_0$ for any {vertex-weighted} directed graph $G(V_n, E_n; \mathscr{A}_n)$ in the graph sequence $\boldsymbol{G}$, where $C$ and $C^{(0)}$ are the matrices whose $(i,j)$th components are $\sigma_-(\mathbf{X}_i) \sigma_+(\mathbf{X}_j) 1(j \to i)$ and $\sigma_-(\mathbf{X}_i) |g(0,\mathbf{X}_j)| 1(j \to i)$, respectively.

Now let $\mathbf{\tilde R}^{(k)}$ be the vector whose $i$th components is given by
\begin{align*}
\tilde R_i^{(k)} &= \Phi\left( \mathbf{X}_i, \zeta_i^{(k-1)}, \left\{ g(\mathcal{R}_{\theta(j)}^{(k-1)}, \mathbf{X}_j), \xi_{j,i}^{(k-1)}: j \to i \right\} \right). 
\end{align*}
Then, the triangle inequality gives for any $k \geq 1$, 
\begin{align*}
 &\left( \frac{1}{n} \mathbb{E}_n\left[ \left\|  \mathcal{J}^{(k)} (\mathbf{R}^{(k)}  - \boldsymbol{\mathcal{R}}^{(k)}) \right\|_p^p  \right] \right)^{1/p} \\
 &\leq  \left( \frac{1}{n} \mathbb{E}_n\left[ \left\|  \mathcal{J}^{(k)} (\mathbf{R}^{(k)} - \mathbf{\tilde R}^{(k)})  \right\|_p^p  \right] \right)^{1/p} +   \left( \frac{1}{n} \mathbb{E}_n\left[ \left\|  \mathcal{J}^{(k)} \mathbf{\tilde R}^{(k)}   \right\|_p^p  \right] \right)^{1/p}  + \left( \frac{1}{n} \mathbb{E}_n\left[ \left\|  \mathcal{J}^{(k)} \boldsymbol{\mathcal{R}}^{(k)} \right\|_p^p  \right] \right)^{1/p}.
 \end{align*}
Now condition on $\mathbf{R}^{(k-1)}$ and $\boldsymbol{\mathcal{R}}^{(k-1)}$ and use {Assumption~\ref{A.PhiMap}(4)(i)} to obtain that
\begin{align*}
 \left( \frac{1}{n} \mathbb{E}_n\left[ \left\|  \mathcal{J}^{(k)} (\mathbf{R}^{(k)} - \mathbf{\tilde R}^{(k)})  \right\|_p^p  \right] \right)^{1/p} &\leq  \left( \frac{1}{n} \mathbb{E}_n\left[ \left\|  C \cdot \left| \mathbf{R}^{(k-1)} - \boldsymbol{\mathcal{R}}^{(k-1)} \right|  \right\|_p^p  \right] \right)^{1/p} \\
 &\leq \left( \frac{1}{n} \mathbb{E}_n\left[ \left\|  C \right\|_p^p \left\| \mathbf{R}^{(k-1)} - \boldsymbol{\mathcal{R}}^{(k-1)} \right\|_p^p  \right] \right)^{1/p} \\
 &\leq K \left( \frac{1}{n} \mathbb{E}_n\left[ \left\| \mathbf{R}^{(k-1)} - \boldsymbol{\mathcal{R}}^{(k-1)} \right\|_p^p  \right] \right)^{1/p}. 
\end{align*}
Next, let $C_*$ be the matrix whose $(i,j)$th component is $\sigma_-(\mathbf{X}_i) 1(j \to i)$, let $\boldsymbol{\beta}$, $\mathbf{g}_0$ and $\mathbf{\tilde V}^{(k-1)}$ be the vectors whose $i$th components are $\beta(\mathbf{X}_i)$, $|g(0,\mathbf{X}_i)|$ and $|g(\mathcal{R}^{(k-1)}_{\emptyset(i)}, \mathbf{X}_i)|$, respectively,  and use Assumption~\ref{A.PhiMap} again to obtain that
\begin{align*}
\left( \frac{1}{n} \mathbb{E}_n\left[ \left\|  \mathcal{J}^{(k)} \mathbf{\tilde R}^{(k)}   \right\|_p^p  \right] \right)^{1/p} &\leq \left( \frac{1}{n} \mathbb{E}_n\left[ \mathbf{E}_n\left[ \left. \left\|  \mathcal{J}^{(k)} \mathbf{\tilde R}^{(k)}   \right\|_p^p \right| \boldsymbol{\mathcal{R}}^{(k-1)} \right]  \right] \right)^{1/p} \\
&\leq \left( \frac{1}{n} \mathbb{E}_n\left[ \left\|  \mathcal{J}^{(k)} (C_* \mathbf{\tilde V}^{(k-1)} + \boldsymbol{\beta})    \right\|_p^p  \right] \right)^{1/p} \\
&\leq \left( \frac{1}{n} \mathbb{E}_n\left[ \left\|  \mathcal{J}^{(k)} C_* | \mathbf{\tilde V}^{(k-1)}  - \mathbf{g}_0|    \right\|_p^p  \right] \right)^{1/p} \\
&\hspace{5mm} + \left( \frac{1}{n} \mathbb{E}_n\left[ \left\|  \mathcal{J}^{(k)} C_* \mathbf{g}_0   \right\|_p^p  \right] \right)^{1/p} + \left( \frac{1}{n} \mathbb{E}_n\left[ \left\|  \mathcal{J}^{(k)} \boldsymbol{\beta}   \right\|_p^p  \right] \right)^{1/p} \\
&\leq \left( \frac{1}{n} \mathbb{E}_n\left[ \left\|  \mathcal{J}^{(k)} C | \boldsymbol{\mathcal{R}}^{(k-1)}  |    \right\|_p^p  \right] \right)^{1/p} \\
&\hspace{5mm} + \left( \frac{1}{n} \mathbb{E}_n\left[ \left\|  \mathcal{J}^{(k)} C^{(0)} \mathbf{e}   \right\|_p^p  \right] \right)^{1/p} + \left( \frac{1}{n} \mathbb{E}_n\left[ \left\|  \mathcal{J}^{(k)} \boldsymbol{\beta}   \right\|_p^p  \right] \right)^{1/p},
\end{align*}
where $\mathbf{e}$ is the (column) vector of ones in $\mathbb{R}^n$. 

We have thus shown that
\begin{align*}
a_{k,n} &:= \left( \mathbb{E}_n\left[ \left| R_I^{(k)} - \mathcal{R}^{(k)}_\emptyset \right|^p  \right] \right)^{1/p}  \\
&\leq K a_{k-1,n} + \left( \frac{1}{n} \mathbb{E}_n\left[ \left\| (\mathbb{I} - \mathcal{J}^{(k)}) (\mathbf{R}^{(k)} - \boldsymbol{\mathcal{R}}^{(k)}) \right\|_p^p  \right] \right)^{1/p} + \left( \frac{1}{n} \mathbb{E}_n\left[ \left\|  \mathcal{J}^{(k)} \boldsymbol{\mathcal{R}}^{(k)} \right\|_p^p  \right] \right)^{1/p} \\
&\hspace{5mm} + \left( \frac{1}{n} \mathbb{E}_n\left[ \left\|  \mathcal{J}^{(k)} C | \boldsymbol{\mathcal{R}}^{(k-1)}  |    \right\|_p^p  \right] \right)^{1/p} +    \left( \frac{1}{n} \mathbb{E}_n\left[ \left\|  \mathcal{J}^{(k)} C^{(0)} \mathbf{e} \right\|_p^p  \right] \right)^{1/p}  + \left( \frac{1}{n} \mathbb{E}_n\left[ \left\|  \mathcal{J}^{(k)}  \boldsymbol{\beta}  \right\|_p^p  \right] \right)^{1/p}.
\end{align*}

Next, use Corollary~\ref{C.Explicit} to obtain that
$$ \left( \frac{1}{n} \mathbb{E}_n\left[ \left\| (\mathbb{I} - \mathcal{J}^{(k)}) (\mathbf{R}^{(k)} - \boldsymbol{\mathcal{R}}^{(k)}) \right\|_p^p  \right] \right)^{1/p} = \left( \mathbb{E}_n \left[ 1(F_\emptyset^{(k,\epsilon)}) \left| R_I^{(k)} - \mathcal{R}_\emptyset^{(k)} \right|^p \right] \right)^{1/p} \leq H_{k, \epsilon} w(\epsilon),$$
where $w(\epsilon) = H \epsilon^\alpha + Q\epsilon^\gamma + HQ \epsilon^{\alpha+\gamma}$ and $H_{k, \epsilon} < \infty$ is defined in the corollary. Now choose $M = M(\epsilon) > 0$ such that $\max_{0 \leq r \leq k} E\left[ \left(( |\mathcal{R}_\emptyset^{(r)}| - M)^+ \right)^p \right] < \epsilon$ and note that if we let $\mathbf{x} \wedge \mathbf{y}$ denote the vector whose $i$th component is $x_i \wedge y_i$ and $(\mathbf{x} - \mathbf{y})^+$ the one whose $i$th component is $(x_i - y_i)^+$, then
\begin{align*}
&\left( \frac{1}{n} \mathbb{E}_n\left[ \left\|  \mathcal{J}^{(k)} C | \boldsymbol{\mathcal{R}}^{(k-1)}  |    \right\|_p^p  \right] \right)^{1/p}  \\
&\leq  \left( \frac{1}{n} \mathbb{E}_n\left[ \left\|  \mathcal{J}^{(k)} C (| \boldsymbol{\mathcal{R}}^{(k-1)}  | \wedge (M\mathbf{e}))   \right\|_p^p  \right] \right)^{1/p}  +  \left( \frac{1}{n} \mathbb{E}_n\left[ \left\|   C \right\|_p^p \left\| ( | \boldsymbol{\mathcal{R}}^{(k-1)}  | - M\mathbf{e})^+   \right\|_p^p  \right] \right)^{1/p} \\
&\leq M \left( \frac{1}{n} \mathbb{E}_n\left[ \left\|  \mathcal{J}^{(k)} C \mathbf{e}   \right\|_p^p  \right] \right)^{1/p}  + K \left( E\left[ \left( |\mathcal{R}_\emptyset^{(k-1)}| - M)^+ \right)^p \right] \right)^{1/p} \\
&\leq M \left(  \mathbb{E}_n\left[ \left(  J_I^{(k)} \sigma_-(\mathbf{X}_I) \sum_{j \to I} \sigma_+(\mathbf{X}_j)   \right)^p  \right] \right)^{1/p} + K \epsilon. 
\end{align*}

It follows that
$$a_{k,n} \leq K a_{k-1,n} + b_{k,n}(\epsilon),$$
where
\begin{align*}
b_{k,n}(\epsilon) &=  H_{k,\epsilon} w(\epsilon) + K \epsilon + \left( \mathbb{E}_n\left[ J_I^{(k)} |\mathcal{R}_\emptyset^{(k)}|^p \right] \right)^{1/p} +  \left(  \mathbb{E}_n\left[ \left(  J_I^{(k)} \sigma_-(\mathbf{X}_I) \sum_{j \to I} |g(0,\mathbf{X}_j)|   \right)^p  \right] \right)^{1/p}   \\
&\hspace{5mm} +  M  \left(  \mathbb{E}_n\left[ \left(  J_I^{(k)} \sigma_-(\mathbf{X}_I) \sum_{j \to I} \sigma_+(\mathbf{X}_j)   \right)^p  \right] \right)^{1/p} +   \left( \mathbb{E}_n\left[ J_I^{(k)} \beta(\mathbf{X}_I)^p \right] \right)^{1/p}  .
\end{align*}
 Iterating the recursion and noting that $a_{0,n} = 0$ gives
 \begin{align*}
 a_{k,n} &\leq K (K a_{k-2,n} + b_{k-1,n}) + b_{k,n} \leq K^{k} a_{0,n}  + \sum_{r=0}^{k-1} K^r b_{k-r,n} \leq \max_{1 \leq r \leq k} b_{r,n} \sum_{r=0}^{k-1} K^r. 
 \end{align*}

It only remains to compute the limit of $b_{r,n}$ as $n \to \infty$ for each fixed $r \geq 1$. To do this, use dominated convergence to obtain that
$$\mathbb{E}_n\left[ J_I^{(r)} |\mathcal{R}_\emptyset^{(k)}|^p \right] \xrightarrow{P} 0, \qquad n \to \infty,$$
and use Lemma~\ref{L.CondDp} in combination with Assumption~\ref{A.RGmodels} to obtain that
$$\mathbb{E}_n\left[ J_I^{(r)} \beta(\mathbf{X}_I)^p \right] \xrightarrow{P} 0, \qquad n \to \infty.$$
Finally, use Lemma~\ref{L.DecoupledEvents} to obtain that
$$\left(  \mathbb{E}_n\left[ \left(  J_I^{(k)} \sigma_-(\mathbf{X}_I) \sum_{j \to I} \sigma_+(\mathbf{X}_j)   \right)^p  \right] \right)^{1/p} +   \left(  \mathbb{E}_n\left[ \left(  J_I^{(k)} \sigma_-(\mathbf{X}_I) \sum_{j \to I} |g(0,\mathbf{X}_j) |  \right)^p  \right] \right)^{1/p} \xrightarrow{P} 0,$$ as $n \to \infty$. We conclude that
$$\limsup_{n \to \infty} \, b_{r,n}(\epsilon) \leq H_{r,\epsilon} w(\epsilon) + K \epsilon,$$
which in turn yields
$$\limsup_{n \to \infty} \, a_{k,n} \leq \max_{1 \leq r \leq k} ( H_{r,\epsilon} w(\epsilon) + K \epsilon) \sum_{r=0}^{k-1} K^r.$$
Taking $\epsilon \downarrow 0$ completes the proof. 
\end{proof}

{Next, we prove Part (B) of Theorem~\ref{T.Main}.}

{\begin{lemma}
Under the conditions of Theorem~\ref{T.Main}, for any fixed $m,k \geq 1$, $\{I_j: 1 \leq j \leq m\}$ i.i.d.~uniformly chosen in $V_n$, independent of $\mathscr{G}_n$, and for any set of bounded and continuous functions $\{f_j: 1 \leq j \leq m\}$ on $\mathbb{R}^{k+1}$, we have
$$E\left[ \prod_{j=1}^m f_j(R_{I_j}^{(0)}, \dots, R_{I_j}^{(k)}) \right] \to \prod_{j = 1}^m E[f_j(\mathcal{R}_\emptyset^{(0)}, \dots, \mathcal{R}_\emptyset^{(k)})], \qquad n \to \infty.$$
\end{lemma}}

\begin{proof}
By Definition~\ref{D.StrongCoupling}, there exists a set of i.i.d.~copies of the marked Galton-Watson process $\mathcal{T}^{(k)}(\boldsymbol{\mathcal{X}})$, denoted $\{ \mathcal{T}_{\emptyset(I_j)}^{(k)}(\boldsymbol{\mathcal{X}}): 1 \leq j \leq m\}$, whose roots correspond to the vertices $\{I_j: 1 \leq j \leq m\}$, and such that for any $\epsilon > 0$, $\mathbb{P}_n\left( \bigcap_{j =1}^m C_{I_j}^{(k,\epsilon)} \right) \xrightarrow{P} 1$ as $n \to \infty$, where the events $C_{I_j}^{(k,\epsilon)}$ ensure that  $\mathcal{G}_{I_j}^{(k)} \simeq \mathcal{T}^{(k)}_{\emptyset(I_j)}$ for all $1 \leq j \leq m$ and their corresponding marks are within $\epsilon$ distance from each other (see Definition~\ref{D.StrongCoupling}). It follows that on the event $\mathcal{E}_n = \bigcap_{j=1}^m C_{I_j}^{(k,\epsilon)}$ the subgraphs $\{ \mathcal{G}_{I_j}^{(k)}: 1 \leq j \leq m\}$ of $G(V_n,E_n; \mathscr{A}_n)$ share no vertices. Hence, by coupling the noises as described at the beginning of Section~\ref{S.Proofs}, we can construct each of the trajectories $\{ (\mathcal{R}_{\emptyset(I_j)}^{(0)}, \dots, \mathcal{R}_{\emptyset(I_j)}^{(k)}): 1 \leq j \leq m\}$ on the coupled trees in such a way that they are conditionally independent given $\{ \mathcal{G}_{I_j}^{(k)}: 1 \leq j \leq m\}$.  Therefore, for these constructions we have
\begin{align*}
&\left| E\left[ \prod_{j=1}^m f_j(R_{I_j}^{(0)}, \dots, R_{I_j}^{(k)}) \right] - \prod_{j = 1}^m E[f_j(\mathcal{R}_\emptyset^{(0)}, \dots, \mathcal{R}_\emptyset^{(k)})] \right| \\
&\leq \left| E\left[ \left( \prod_{j=1}^m f_j(R_{I_j}^{(0)}, \dots, R_{I_j}^{(k)}) - \prod_{j=1}^m f_j(\mathcal{R}_{\emptyset(I_j)}^{(0)}, \dots, \mathcal{R}_{\emptyset(I_j)}^{(k)}) \right)  1(\mathcal{E}_n) \right] \right| + 2 \prod_{j=1}^m \sup_{\mathbf{x} \in \mathbb{R}^{k+1}} |f_j(\mathbf{x})| P(\mathcal{E}_n^c)  \\
&\leq    \sum_{j=1}^m \prod_{t\neq j} \sup_{\mathbf{x} \in \mathbb{R}^{k+1}} |f_t(\mathbf{x})|  E\left[ \left|  f_j(R_{I_j}^{(0)}, \dots, R_{I_j}^{(k)}) - f_j(\mathcal{R}_{\emptyset(I_j)}^{(0)}, \dots, \mathcal{R}_{\emptyset(I_j)}^{(k)}) \right| \right] + 2 \prod_{j=1}^m \sup_{\mathbf{x} \in \mathbb{R}^{k+1}} |f_j(\mathbf{x})| P(\mathcal{E}_n^c),
\end{align*}
which converges to zero as $n \to \infty$ since $\max_{0\leq r\leq k} |R_{I_j}^{(r)} - \mathcal{R}_{\emptyset(I_j)}^{(r)} | \xrightarrow{P} 0$ for each $1\leq j \leq m$ (see Remark~\ref{R.Observations}) and each $f_j$ is bounded and continuous.
\end{proof}

We now prove {Part (C)} of Theorem~\ref{T.Main}, which relates to the distributional fixed-point equation. 

\begin{theorem}
Let $\nu_k(\cdot) = P\left( \mathcal{R}^{(k)}_\emptyset \in \cdot \right)$. Then, if $c^p = E\left[ (\mathcal{N}_1 \sigma_+({\boldsymbol{\mathcal{X}}_1}) \sigma_-({\boldsymbol{\mathcal{X}}_1}))^p \right] \in (0, 1)$ and $r_0^p = E[ |R^{(0)}|^p] < \infty$, where $R^{(0)}$ is distributed according to $\mu_0$, then, there exists a probability measure $\nu$ on $\mathbb{R}$ such that
$$d_p(\nu_k, \nu) \to 0, \qquad k \to \infty.$$
Moreover, $\nu$ is the probability measure of a random variable $\mathcal{R}^*$ that satisfies:
\begin{align*}
\mathcal{R}^* &= \Phi\left( {\boldsymbol{\mathcal{X}}_\emptyset}, \zeta, \{ \mathcal{V}_j, \xi_j: 1 \leq j \leq \mathcal{N}_\emptyset \} \right), 
\end{align*}
with the $\{\mathcal{V}_j\}$ i.i.d.~copies of $\mathcal{V}$, independent of ${\boldsymbol{\mathcal{X}}_\emptyset}$ and of $(\zeta, \{ \xi_j: j \geq 1\})$, and $\mathcal{V}$ the attracting endogenous solution to the distributional fixed-point equation:
\begin{align*}
\mathcal{V} &\stackrel{\mathcal{D}}{=} \Psi\left( {\boldsymbol{\mathcal{X}}_1}, \zeta, \{ \mathcal{V}_j, \xi_j: 1 \leq j \leq \mathcal{N}_1 \} \right),
\end{align*}
where the $\{ \mathcal{V}_j \}$ i.i.d.~copies of $\mathcal{V}$, independent of $({\boldsymbol{\mathcal{X}}_1}, \zeta, \{ \xi_j: j \geq 1\})$.
\end{theorem}

\begin{proof}
Fix $k \geq 1$ and define
\begin{align*}
\mathcal{V}_\mathbf{i}^{(0)} &= g(0, {\boldsymbol{\mathcal{X}}_{\bf i}}), \quad \mathbf{i} \in \mathcal{A}_k, \\
\mathcal{V}_\mathbf{i}^{(r)} &= \Psi \left({\boldsymbol{\mathcal{X}}_{\bf i}}, \zeta_{\bf i}, \{ \mathcal{V}_{({\bf i},j)}^{(r)}, \xi_{({\bf i},j)}: 1 \leq j \leq \mathcal{N}_{\bf i}  \} \right), \quad \mathbf{i} \in \mathcal{A}_{k-r}, \, 1 \leq r < k, \\
\mathcal{R}_{\emptyset}^{(k)} &= \Phi\left( {\boldsymbol{\mathcal{X}}_\emptyset}, \zeta_\emptyset, \{ \mathcal{V}_{j}^{(k-1)}, \xi_j: 1 \leq j \leq \mathcal{N}_\emptyset  \} \right),
\end{align*}
which corresponds to taking $\nu_0$ to be the Dirac measure at zero. Let $\eta_r(\cdot) = P(\mathcal{V}_1^{(r)} \in \cdot)$ for $r \geq 0$.  Now fix $m \geq 1$ and note that
\begin{align*}
&\left( E\left[ \left| \mathcal{R}_\emptyset^{(k)} - \mathcal{R}_\emptyset^{(k+m)} \right|^p \right] \right)^{1/p} \\
&= \left( E\left[ E\left[ \left. \left| \mathcal{R}_\emptyset^{(k)} - \mathcal{R}_\emptyset^{(k+m)} \right|^p \right| {\boldsymbol{\mathcal{X}}_\emptyset}, \{ \mathcal{V}_j^{(k-1)}, \mathcal{V}_j^{(k-1+m)}: j \in \mathcal{A}_1 \} \right] \right] \right)^{1/p} \\
&\leq  \left( E\left[ \left( \sigma_-({\boldsymbol{\mathcal{X}}_\emptyset}) \sum_{j=1}^{\mathcal{N}_\emptyset} \left| \mathcal{V}_j^{(k-1)} - \mathcal{V}_j^{(k-1+m)} \right|  \right)^p \right] \right)^{1/p} \\
&=  \left( E\left[ \sigma_-({\boldsymbol{\mathcal{X}}_\emptyset})^p E\left[ \left. \left(  \sum_{j=1}^{\mathcal{N}_\emptyset} \left| \mathcal{V}_j^{(k-1)} - \mathcal{V}_j^{(k-1+m)} \right|  \right)^p \right| \mathcal{N}_\emptyset \right] \right] \right)^{1/p} \\
&\leq  \left( E\left[ \sigma_-({\boldsymbol{\mathcal{X}}_\emptyset})^p \left(  \sum_{j=1}^{\mathcal{N}_\emptyset} \left( E\left[ \left.   \left| \mathcal{V}_j^{(k-1)} - \mathcal{V}_j^{(k-1+m)} \right|^p \right| \mathcal{N}_\emptyset \right] \right)^{1/p} \right)^p \right] \right)^{1/p} \\
&=  \left( E\left[ (\sigma_-({\boldsymbol{\mathcal{X}}_\emptyset}) \mathcal{N}_\emptyset )^p \right] \right)^{1/p} \left( E\left[   \left| \mathcal{V}_1^{(k-1)} - \mathcal{V}_1^{(k-1+m)} \right|^p  \right] \right)^{1/p}. 
\end{align*}

Similarly, {Assumption~\ref{A.PhiMap}(1-2)} also yields
\begin{align*}
&\left( E\left[    \left| \mathcal{V}_1^{(k-1)} - \mathcal{V}_1^{(k-1+m)} \right|^p  \right] \right)^{1/p} \\
&\leq  \left( E\left[ (\sigma_-({\boldsymbol{\mathcal{X}}_1}) \sigma_+({\boldsymbol{\mathcal{X}}_1}) \mathcal{N}_1)^p \right] \right)^{1/p} \left( E\left[   \left| \mathcal{V}_{(1,1)}^{(k-2)} - \mathcal{V}_{(1,1)}^{(k-2+m)} \right|^p  \right] \right)^{1/p} \\
&= c \left( E\left[  \left| \mathcal{V}_1^{(k-2)} - \mathcal{V}_1^{(k-2+m)} \right|^p  \right] \right)^{1/p} \\
&\leq c^{k-1} \left( E\left[  \left| \mathcal{V}_1^{(0)} - \mathcal{V}_1^{(m)} \right|^p  \right] \right)^{1/p} \\
&\leq c^{k-1} \left(  \left( E\left[    \left| \mathcal{V}_1^{(0)} \right|^p  \right] \right)^{1/p} +  \left( E\left[  \left|  \mathcal{V}_1^{(m)} \right|^p  \right] \right)^{1/p} \right).
\end{align*}

Now use Lemma~\ref{L.Moments} to obtain that 
\begin{align*}
\sup_{m \geq 0} \left( E\left[  \left(  \left|  \mathcal{V}_1^{(m)} \right|  \right)^p  \right] \right)^{1/p}& \leq (1-c)^{-1} \left( \left( E \left[ (\sigma_+({\boldsymbol{\mathcal{X}}_1}) \beta({\boldsymbol{\mathcal{X}}_1}))^p \right] \right)^{1/p}  +  \left( E \left[ |g(0,{\boldsymbol{\mathcal{X}}_1})|^p \right] \right)^{1/p} \right) \\
&\hspace{5mm} + \left( E\left[ |g(0, {\boldsymbol{\mathcal{X}}_1})|^p \right] \right)^{1/p} < \infty
\end{align*}
and conclude that
\begin{align*}
&\sup_{m \geq 0} \left( E\left[ \left| \mathcal{R}_\emptyset^{(k)} - \mathcal{R}_\emptyset^{(k+m)} \right|^p \right] \right)^{1/p} \\
&\leq  \left( E\left[ (\sigma_-({\boldsymbol{\mathcal{X}}_\emptyset}) \mathcal{N}_\emptyset )^p \right] \right)^{1/p} \left( (1-c)^{-1} \left(  \left( E \left[ (\sigma_+({\boldsymbol{\mathcal{X}}_1}) \beta({\boldsymbol{\mathcal{X}}_1}))^p \right] \right)^{1/p} +  \left( E \left[ |g(0,{\boldsymbol{\mathcal{X}}_1})|^p \right] \right)^{1/p}   \right)  \right. \\
&\hspace{5mm} \left. + 2 \left( E\left[ |g(0, {\boldsymbol{\mathcal{X}}_1})|^p \right] \right)^{1/p} \right) c^{k-1} \rightarrow 0, \qquad k \to \infty.
\end{align*}

Therefore, the sequences of random variables $\{ \mathcal{R}_\emptyset^{(k)}: k \geq 0\}$ and $\{ \mathcal{V}_1^{(k)}: k \geq 0\}$ are Cauchy under the $L^p$ norm, and since the $L^p$ norm is a complete metric, there exist random variables $\mathcal{R}^*$ and $\mathcal{V}$ such that
$$\left( E\left[ \left| \mathcal{R}_\emptyset^{(k)} - \mathcal{R}^* \right|^p \right] \right)^{1/p} + \left( E\left[ \left| \mathcal{V}_1^{(k)} - \mathcal{V} \right|^p \right] \right)^{1/p} \to 0, \qquad k \to \infty.$$
Let $\nu(\cdot) = P(\mathcal{R}^* \in \cdot)$ and $\eta(\cdot) = P(\mathcal{V} \in \cdot)$, then the above also implies that
$$d_p(\eta_k, \eta) + d_p(\nu_k, \nu) \to 0, \qquad k \to \infty.$$

Continuity of the map $\Psi$ gives that $\eta$ solves the distributional fixed-point equation:
\begin{equation} \label{eq:ProofSFPE}
\mathcal{V} \stackrel{\mathcal{D}}{=}  \Psi\left( {\boldsymbol{\mathcal{X}}_1}, \zeta, \{ \mathcal{V}_j, \xi_j: 1 \leq j \leq \mathcal{N}_1 \} \right),
\end{equation}
where the $\{ \mathcal{V}_j \}$ i.i.d.~copies of $\mathcal{V}$, independent of $({\boldsymbol{\mathcal{X}}_1}, \zeta, \{ \xi_j: j \geq 1\})$, and the explicit construction of $\mathcal{R}^{(k)}_\emptyset$ in terms of the $\{\mathcal{V}_j^{(k)}: j \geq 1\}$ implies that $\nu$ is the probability measure of
$$\mathcal{R}^* = \Phi\left( {\boldsymbol{\mathcal{X}}_\emptyset}, \zeta, \{ \mathcal{V}_j, \xi_j: 1 \leq j \leq \mathcal{N}_\emptyset \} \right),$$
with the $\{\mathcal{V}_j\}$  i.i.d.~copies of $\mathcal{V}$, independent of $({\boldsymbol{\mathcal{X}}_\emptyset}, \zeta, \{ \xi_j: j \geq 1\})$. 

To see that the limiting measure $\nu$ is the same for any initial distribution $\mu_0$, note that if $\tilde \nu_k$ is the probability measure of $\mathcal{R}_\emptyset^{(k)}$ when the $\{\mathcal{R}_\mathbf{i}^{(0)}: \mathbf{i} \in \mathcal{A}_k\}$ are chosen according to $\tilde \nu_0 = \mu_0$, the same computations used above give
\begin{align*}
d_p(\nu_k, \tilde \nu_k) &\leq \left( E\left[ \sigma_-({\boldsymbol{\mathcal{X}}_\emptyset}) \mathcal{N}_\emptyset)^p \right] \right)^{1/p} c^{k-1} \left( \left( E\left[ |g(0,{\boldsymbol{\mathcal{X}}_1})|^p \right] \right)^{1/p} + \left( E\left[ \left|g(R^{(0)},{\boldsymbol{\mathcal{X}}_1}) \right|^p \right] \right)^{1/p}  \right) \\
&\leq \left( E\left[ \sigma_-({\boldsymbol{\mathcal{X}}_\emptyset}) \mathcal{N}_\emptyset)^p \right] \right)^{1/p} c^{k-1} \left( 2\left( E\left[ |g(0,{\boldsymbol{\mathcal{X}}_1})|^p \right] \right)^{1/p} + \left( E\left[ \sigma_+({\boldsymbol{\mathcal{X}}_1})^p \right] \right)^{1/p} r_0  \right) \to 0
\end{align*}
as $k \to \infty$, which implies that $d_p(\tilde \nu_k, \nu) \to 0$ as $k \to \infty$. Hence, $\eta$ is the unique solution to \eqref{eq:ProofSFPE} in the space of probability measures on $\mathbb{R}$ with finite $p$th moment. 

Finally, to see that $\eta$ is an endogenous solution to \eqref{eq:ProofSFPE}, note that the sequence $\{ \mathcal{V}^{(k)}_1: k \geq 0\}$ is measurable with respect to $\mathcal{G}_\mathcal{T}$, so if $\mathcal{V}$ denotes its $L^p$ limit, then, by Jensen's inequality,
\begin{align*}
E\left[ \left| \mathcal{V}_1^{(k)} - E\left[ \left. \mathcal{V} \right| \mathcal{G}_\mathcal{T} \right] \right|^p \right] &\leq E\left[   E\left[ \left. \left| \mathcal{V}_1^{(k)} - \mathcal{V} \right|^p \right| \mathcal{G}_\mathcal{T} \right]  \right] = E\left[    \left| \mathcal{V}_1^{(k)} - \mathcal{V} \right|^p   \right] \to 0, \qquad k \to \infty.
\end{align*}
It follows  that $E\left[ \left. \mathcal{V} \right| \mathcal{G}_\mathcal{T} \right]  = \mathcal{V}$, and therefore, $\eta$ is endogenous. 
\end{proof}

It only remains to prove Corollary~\ref{C.MainContraction}, for which it is useful to start with the part that relates to the Markov chain $\{ \mathbf{R}^{(k)}: k \geq 0\}$ on $\mathbb{R}^n$.

\begin{lemma} \label{L.VectorMC}
Suppose that Assumption~\ref{A.PhiMap} holds with Assumption~\ref{A.PhiMap}(4)(i) and  $\| C \|_p \leq K < 1$. Define $\lambda_{k,n}(\cdot) = \mathbf{P}_n\left(\mathbf{R}^{(k)} \in \cdot\right)$. Then, provided $r_0^p = E[|R_1^{(0)}|^p]  < \infty$, there exists a probability measure $\lambda_n$ on $\mathbb{R}^n$ such that 
$$d_p(\lambda_{k,n}, \lambda_n) \to 0 \quad \mathbf{P}_n\text{-a.s.}, \qquad k \to \infty.$$
Moreover,  there exists a random vector $\mathbf{R}$ distributed according to $\lambda_n$ such that
$$\left( \mathbf{E}_n\left[ \left\| \mathbf{R}^{(k)} - \mathbf{R} \right\|_p^p \right] \right)^{1/p} \leq \| C \|_p^k \left( 2  r_0 n^{1/p} + \| C^{(0)} \mathbf{e} + \boldsymbol{\beta} \|_p \cdot \frac{1}{1- \| C\|_p} \right),$$
where $C^{(0)}$ is the matrix whose $(i,j)$th component is $\sigma_-(\mathbf{X}_i) 1(j \to i)|g(0, \mathbf{X}_i)|$, $\boldsymbol{\beta}$ is the vector whose $i$th component is $\beta(\mathbf{X}_i)$ and $\mathbf{e}$ is the vector of ones in $\mathbb{R}^n$,
\end{lemma} 

\begin{proof}
To start, fix $m \geq 1$ and sample $\mathbf{\tilde R}^{(0,m)}$ according to $\lambda_{m,n}$, independently of the noises $\{ \boldsymbol{\zeta}^{(r)}, \boldsymbol{\xi}^{(r)}: 0 \leq r \leq k \}$. Next, construct $\{ \mathbf{\tilde R}^{(r,m)}: 1 \leq r \leq k \}$ according to the recursion:
$$\tilde R_i^{(r+1,m)} = \Phi\left( \mathbf{X}_i, \zeta_i^{(r)}, \left\{ g(\tilde R_j^{(r,m)}, \mathbf{X}_j), \xi_{j,i}^{(r)}: j \to i \right\} \right), \qquad i \in V_n.$$
Note that $\mathbf{\tilde R}^{(r,m)}$ is distributed according to $\lambda_{r+m,n}$ for each $0 \leq r \leq k$.  

Hence,
\begin{align*}
d_p(\lambda_{k,n}, \lambda_{k+m,n})^p &\leq \mathbf{E}_n\left[ \left\| \mathbf{R}^{(k)} - \mathbf{\tilde R}^{(k,m)} \right\|_p^p \right] \\
&= \mathbf{E}_n\left[ E\left[ \left. \left\| \mathbf{R}^{(k)} - \mathbf{\tilde R}^{(k,m)} \right\|_p^p \right| \mathscr{G}_n, \mathbf{R}^{(k-1)}, \mathbf{\tilde R}^{(k-1,m)} \right] \right] \\
&\leq \mathbf{E}_n\left[ \sum_{l=1}^n \left( \sum_{j\to l} \sigma_-(\mathbf{X}_l) \sigma_+(\mathbf{X}_j) \left| R_j^{(k-1)} - \tilde R_j^{(k-1,m)}  \right| \right)^p \right] \\
&= \mathbf{E}_n\left[ \sum_{l=1}^n \left( \sum_{j=1}^n C_{l,j} \left| R_j^{(k-1)} - \tilde R_j^{(k-1,m)}  \right| \right)^p \right] \\
&\leq  \mathbf{E}_n\left[  \| C \|_p^p  \left\|  \mathbf{R}^{(k-1)} -  \mathbf{\tilde R}^{(k-1,m)} \right\|_p^p \right]  \\
&\leq \left( \| C \|_p^p \right)^{k}  \mathbf{E}_n\left[  \left\| \mathbf{R}^{(0)} -  \mathbf{\tilde R}^{(0,m)} \right\|_p^p \right] \\
&\leq  \| C \|_p^{pk}  \left( \left( \mathbf{E}_n \left[ \left\|  \mathbf{R}^{(0)} \right\|_p^p \right] \right)^{1/p} + \left( \mathbf{E}_n\left[ \left\|  \mathbf{\tilde R}^{(0,m)} \right\|_p^p  \right] \right)^{1/p}  \right)^p \\
&=  \| C \|_p^{pk}  \left( r_0 n^{1/p} + \left( \mathbf{E}_n\left[ \left\|  \mathbf{R}^{(m)} \right\|_p^p  \right] \right)^{1/p}  \right)^p .
\end{align*}

Moreover, note that {Assumption~\ref{A.PhiMap}(2-3)} gives that
\begin{align*}
\left( \mathbf{E}_n\left[ \left| R_i^{(m)}  \right|^p \right] \right)^{1/p} &\leq \left( \mathbf{E}_n\left[ \mathbf{E}_n\left[ \left. \left| R_i^{(m)}  \right|^p \right| {\bf R}^{(m-1)} \right]  \right] \right)^{1/p} \\
&\leq \left( \mathbf{E}_n\left[ \left( \sum_{j \to i} \sigma_-({\bf X}_i)  | V_j^{(m-1)} | + \beta({\bf X}_i) \right)^p  \right] \right)^{1/p} \\
&\leq \sigma_-({\bf X}_i)  \sum_{j \to i}  \left( \mathbf{E}_n\left[ \left|V_j^{(m-1)} \right|^p \right] \right)^{1/p} + \beta({\bf X}_i) \\
&= \sigma_-({\bf X}_i)  \sum_{j \to i}  \left( \mathbf{E}_n\left[ \left| g(R_j^{(m-1)}, \mathbf{X}_j) \right|^p \right] \right)^{1/p} + \beta({\bf X}_i) \\
&\leq \sigma_-({\bf X}_i)  \sum_{j \to i}   \sigma_+(\mathbf{X}_j) \left( \mathbf{E}_n\left[ \left| R_j^{(m-1)} \right|^p \right] \right)^{1/p} + \sigma_-(\mathbf{X}_i) \sum_{j\to i} |g(0, \mathbf{X}_j)| + \beta({\bf X}_i). 
\end{align*}
If we let $\mathbf{\bar R}^{(m)}$ denote the vector whose $i$th component is $\left( \mathbf{E}_n\left[ \left| R_i^{(m)} \right|^p \right] \right)^{1/p}$, then, 
\begin{align*}
\mathbf{\bar R}^{(m)}  &\leq  C \mathbf{ \bar R}^{(m-1)} + C^{(0)} \mathbf{e} + \boldsymbol{\beta}   \\
&\leq C^2  \mathbf{ \bar R}^{(m-2)} + C(C^{(0)} \mathbf{e} + \boldsymbol{\beta} ) + C^{(0)} \mathbf{e} + \boldsymbol{\beta}  \\
&\leq C^m \mathbf{\bar R}^{(0)} + \sum_{r=0}^{m-1} C^r (C^{(0)} \mathbf{e} + \boldsymbol{\beta})  .
\end{align*}
Hence, since $\mathbf{\bar R}^{(0)}  = r_0 \mathbf{e}$ with $\mathbf{e}$ the vector of ones, 
\begin{align*}
\left( \mathbf{E}_n\left[ \left\|  \mathbf{R}^{(m)} \right\|_p^p  \right] \right)^{1/p} &=  \left\|  \mathbf{\bar R}^{(m)} \right\|_p \leq \left\| C^m r_0 \mathbf{e} + \sum_{r=0}^{m-1} C^r (C^{(0)} \mathbf{e} + \boldsymbol{\beta}) \right\|_p \\
&\leq \| C\|_p^m r_0 n^{1/p} + \sum_{r=0}^{m-1} \| C \|_p^r \| C^{(0)} \mathbf{e} + \boldsymbol{\beta} \|_p. 
\end{align*}

Therefore,
\begin{align*}
\sup_{m \geq 0} d_p(\lambda_{k,n}, \lambda_{k+m,n}) &\leq \| C \|_p^k \left( 2  r_0 n^{1/p} + \| C^{(0)} \mathbf{e}+ \boldsymbol{\beta} \|_p \cdot \frac{1}{1- \| C\|_p} \right) . 
\end{align*}
It follows that the sequence $\{\lambda_{k,n}: k \geq 0\}$ is Cauchy under $d_p$, so by the properties of the Wasserstein metrics, there exists a measure $\lambda_n$ on $\mathbb{R}^n$ such that 
$$\lim_{k \to \infty} d_p(\lambda_{k,n}, \lambda_n) = 0 \qquad \mathbf{P}_n\text{-a.s. }$$

The random vector $\mathbf{R}$ in the statement of the lemma can be obtained by sampling $\mathbf{\tilde R}^{(0,*)}$ according to $\lambda_n$, independently of everything else, then constructing the random vector $\mathbf{\tilde R}^{(k,*)}$ as above, and setting $\mathbf{R} = \mathbf{\tilde R}^{(k,*)}$.
\end{proof} 

We can now prove the second statement in Corollary~\ref{C.MainContraction}, which establishes that $d_p(\mu_n, \nu) \xrightarrow{P} 0$ as $n \to \infty$.

\begin{proof}[Proof of Corollary~\ref{C.MainContraction}]
Use Lemma~\ref{L.VectorMC} to obtain that there exists a measure $\lambda_n$ such that 
$$\lim_{k \to \infty} d_p(\lambda_{k,n}, \lambda_n) = 0 \qquad \mathbf{P}_n\text{-a.s. }$$
By Lemma~\ref{L.VectorMC} again, there exists a vector $\mathbf{R} = (R_1, \dots, R_n)$, distributed according to $\lambda_n$, such that
\begin{align*}
d_p(\mu_{k,n}, \mu_n) &\leq \left( \mathbb{E}_n \left[ \left| R^{(k)}_I - R_I \right|^p \right] \right)^{1/p} = \left( \frac{1}{n}  \mathbb{E}_n \left[ \left\| \mathbf{R}^{(k)} - \mathbf{R} \right\|_p^p \right] \right)^{1/p}  \\
&= \left( \frac{1}{n}  \mathbb{E}_n \left[ \mathbf{E}_n\left[  \left\| \mathbf{R}^{(k)} - \mathbf{R} \right\|_p^p \right]  \right] \right)^{1/p} \\
&\leq \left( \frac{1}{n}  \mathbb{E}_n \left[ \left( \| C \|_p^k \left( 2 r_0 n^{1/p} + \| C^{(0)} \mathbf{e} + \boldsymbol{\beta} \|_p \cdot \frac{1}{1- \| C\|_p} \right) \right)^p  \right] \right)^{1/p} \\
&\leq \frac{K^k }{n^{1/p}} \left( 2r_0 n^{1/p}+  \frac{1}{1-K }   \left( \mathbb{E}_n\left[  \left\| C^{(0)} \mathbf{e}  + \boldsymbol{\beta} \right\|_p^p \right] \right)^{1/p}   \right) \\
&= K^k  \left( 2r_0 +  \frac{1}{1- K }   \left( \frac{1}{n} \mathbb{E}_n\left[  \left\| C^{(0)} \mathbf{e} \right\|_p^p \right] \right)^{1/p} +  \frac{1}{1- K }   \left( \frac{1}{n} \mathbb{E}_n\left[  \left\|  \boldsymbol{\beta} \right\|_p^p \right] \right)^{1/p}  \right) .
\end{align*}
 
Note that by Assumption~\ref{A.RGmodels} we have
$$ \frac{1}{n} \mathbb{E}_n\left[  \left\|  \boldsymbol{\beta} \right\|_p^p \right] = \mathbb{E}_n\left[ \beta(\mathbf{X}_I)^p \right] \xrightarrow{P} E[\beta({\boldsymbol{\mathcal{X}}_\emptyset})^p ] < \infty,$$
as $n \to \infty$. And by Assumption~\ref{A.PhiMap}(4)(i) we have
$$ \frac{1}{n} \mathbb{E}_n\left[  \left\| C^{(0)} \mathbf{e} \right\|_p^p \right] \leq \frac{1}{n} \mathbb{E}_n\left[  \left\| C^{(0)}  \right\|_p^p \| \mathbf{e}\|_p^p  \right] \leq K_0 < \infty.$$

We conclude that
$$d_p(\mu_{k,n}, \mu_n) \to 0 \quad \mathbf{P}_n\text{-a.s.}, \qquad k \to \infty,$$
and Theorem~\ref{T.Main} gives
$$\lim_{k \to \infty} \lim_{n \to \infty} d_p(\mu_{k,n}, \nu) \leq \lim_{k \to \infty} \lim_{n \to \infty} \left( d_p(\mu_{k,n}, \nu_k) + d_p(\nu_k, \nu) \right) = \lim_{k \to \infty} d_p(\nu_k, \nu) = 0.$$
This completes the proof. 
\end{proof}

\section*{Acknowledgment}
We thank an anonymous reviewer whose comments and suggestions helped improve and clarify this manuscript.

\section*{Appendix}

This appendix contains a brief description of the two families of directed random graphs for which a strong coupling, as in Definiton~\ref{D.StrongCoupling}, is known to exist. The precise conditions under which a strong coupling exists are given in \cite{Olvera_21}, and generally involve only finite first moments for the empirical distribution of the latent variables in the vertex attributes $\mathbf{a}_i$, making them suitable for modeling scale-free (graphs whose degree distribution(s) follow(s) a power-law) real-world networks with arbitrarily dependent in-degrees and out-degrees.

\subsection{Directed configuration model} \label{SS.DCM}

One model that produces graphs from any prescribed (graphical) degree sequence is the configuration or pairing model \cite{bollobas, Hofstad1}, which assigns to each vertex in the graph a number of half-edges equal to its target degree and then randomly pairs half-edges to connect vertices. 

We assume that each vertex $i$ in the graph has a degree vector ${\bf d}_i = (d_i^-, d_i^+) \in \mathbb{N} \times \mathbb{N}$, where $d_i^-$ and $d_i^+$ are the in-degree and out-degree of vertex $i$, respectively. In order for us to be able to draw the graph, we assume that the degree sequence $\{ {\bf d}_i: 1 \leq i \leq n\}$ satisfies
$$l_n := \sum_{i=1}^n d_i^- = \sum_{i=1}^n d_i^+.$$
Note that in order for the sum of the in-degrees to be equal to that of the out-degrees, it may be necessary to consider a double sequence $\{ {\bf d}_i^{(n)}: i \geq 1, n \geq 1\}$ rather than a unique sequence; i.e., it may be convenient to allow ${\bf d}_i^{(n)} \neq {\bf d}_i^{(m)}$ for $n\neq m$.

Formally, the DCM can be defined as follows.

\begin{defn}
\label{D.DCM}
Let $\{ {\bf d}_i: 1 \leq i \leq n\}$ be a degree sequence and let $V_n = \{1, 2, \dots, n\}$ denote the nodes in the graph. To each node $i$ assign $d_i^-$ inbound half-edges and $d_i^+$ outbound half-edges. Enumerate all $l_n$ inbound half-edges, respectively outbound half-edges, with the numbers $\{1, 2, \dots, l_n\}$, and let ${\bf x}_n = (x_1, x_2, \dots, x_{l_n})$ be a random permutation of these $l_n$ numbers, chosen uniformly at random from the possible $l_n!$ permutations. The DCM with degree sequence $\{ {\bf d}_i: 1 \leq i \leq n\}$ is the directed graph $G(V_n, E_n)$ obtained by pairing the $x_i$th outbound half-edge with the $i$th inbound half-edge.
\end{defn}

We point out that instead of generating the permutation ${\bf x}_n$ of the outbound half-edges up front, one can construct the graph one vertex at a time, by pairing each of the inbound half-edges with an outbound half-edge, randomly chosen with equal probability from the set of unpaired outbound half-edges. 

We emphasize that the DCM is in general a multi-graph, that is, it can have self-loops and multiple edges in the same direction. However, provided the pairing process does not create self-loops or multiple edges, the resulting graph is uniformly chosen among all graphs having the prescribed degree sequence.  If one chooses this degree sequence according to a power-law, one immediately obtains a scale-free graph.  It was shown in~\cite{Chen_Olv_13} that the random pairing of inbound and outbound half-edges results in a simple graph with positive probability provided both the in-degree and out-degree distributions possess a finite variance. In this case, one can obtain a simple realization after finitely many attempts, a method we refer to as the {\it repeated} DCM. Furthermore, if the self-loops and multiple edges in the same direction are simply removed, a model we refer to as the {\it erased} DCM, the degree distributions will remain asymptotically unchanged.

For the purposes of this paper, self-loops and multiple edges in the same direction do not affect the existence of strong couplings. For the strong coupling, the target degree vector $\mathbf{d}_i$ is a latent variable that becomes part of the vertex attribute $\mathbf{a}_i$. 

\subsection{Inhomogeneous random digraphs} \label{SS.IRD}

In the spirit of the classical Erd\H os-R\'enyi graph \cite{Erdos, Gilbert, Austin, Janson, Bollobas2, Durrett1}, we assume that whether there is an edge between vertices $i$ and $j$ is determined by a coin-flip, independently of all other edges. Several models capable of producing graphs with inhomogeneous degrees while preserving the independence among edges have been suggested in the  literature, including: the Chung-Lu model \cite{Chunglu, Chunglu2, Chunglu3, Chunglu4}, the Norros-Reittu model (or Poissonian random graph) \cite{Norros, Hofstad1, Esker_Hofs_Hoog}, and the generalized random graph \cite{Hofstad1, Brittonetal, Esker_Hofs_Hoog}, to name a few. In all of these models, the inhomogeneity of the degrees is created by allowing the success probability of each coin-flip to depend on the {\em attributes} of the two vertices being connected; the scale-free property can then be obtained by choosing the attributes according to a power-law.

Start by assigning to each vertex $i \in V_n$ a {\em type} ${\bf w}_i = (w_i^-, w_i^+) \in \mathbb{R}_+\times \mathbb{R}_+ $; the vector $\mathbf{w}_i$  is a latent variable that becomes part of the vertex attribute $\mathbf{a}_i$. The $w_i^-$ and $w_i^+$ will be used to determine how likely vertex $i$ is to have inbound/outbound neighbors. As for the DCM, it may be convenient to consider a double sequence $\{{\bf w}_i^{(n)}:  i \geq 1, \, n \geq 1\}$ rather than a unique sequence. {Let $\mathscr{F}_n =\sigma( {\bf a}_i: 1 \leq i \leq n)$ denote the sigma algebra generated by the vertex attributes}, and  recall that $\mathbb{P}_n(\cdot) = P( \cdot | \mathscr{F}_n)$ denotes the conditional probability given $\mathscr{F}_n$. 

We now define our family of random digraphs using the conditional probability, given $\mathscr{F}_n$, that edge $(i,j) \in E_n$, 
\begin{equation} \label{eq:EdgeProbabilities}
p_{ij}^{(n)} := \mathbb{P}_n \left( (i,j) \in E_n \right) = 1 \wedge  \frac{w_i^+ w_j^-}{\theta n} (1 + \varphi_n({\bf w}_i, {\bf w}_j)) , \qquad 1 \leq i \neq j \leq n,
\end{equation}
where $-1 < \varphi_n({\bf w}_i, {\bf w}_j) = \varphi(n, {\bf w}_i, {\bf w}_j, \mathscr{W}_n)$ a.s.~is a function that may depend on the entire sequence $\mathscr{W}_n := \{{\bf w}_i: 1 \leq i \leq n\}$, on the types of the vertices $(i,j)$, or exclusively on $n$, and $0<\theta < \infty$ satisfies
$$\frac{1}{n} \sum_{i=1}^n (w_i^- + w_i^+) \stackrel{P}{\longrightarrow} \theta, \qquad n \to \infty.$$
In the context of {\cite{Boll_Jan_Rio_07}}, definition \eqref{eq:EdgeProbabilities} corresponds to {a special case of the family of random digraphs having edge probabilities:
$$p_{ij}^{(n)} = \mathbb{P}_n \left( (i,j) \in E_n \right) =\left( 1 \wedge \frac{\kappa(\mathbf{w}_i, \mathbf{w}_j)}{n} \right) (1 + \varphi_n(\mathbf{w}_i, \mathbf{w}_j)), \qquad 1 \leq i \neq j \leq n,$$
where $\kappa: \mathbb{R}_+^2 \times \mathbb{R}_+^2 \to \mathbb{R}_+$ is a nonnegative kernel. Specifically, it corresponds to the so-called rank-1 kernel, i.e., $\kappa({\bf w}_i, {\bf w}_j) = \kappa_+({\bf w}_i) \kappa_-({\bf w}_j)$, with $\kappa_+({\bf w}) = w^+/\sqrt{\theta}$ and $\kappa_-({\bf w}) = w^-/\sqrt{\theta}$.  }

\bibliographystyle{plain}
\bibliography{StochRecBib}

\begin{thebibliography}{10}

\bibitem{acemoglu2015networks}
D.~Acemoglu, A.~Ozdaglar, and A.~Tahbaz-Salehi.
\newblock Networks, shocks, and systemic risk.
\newblock Technical report, National Bureau of Economic Research, 2015.

\bibitem{agirre2014random}
E.~Agirre, O.~L. de~Lacalle, and A.~Soroa.
\newblock Random walks for knowledge-based word sense disambiguation.
\newblock {\em Computational Linguistics}, 40(1):57--84, 2014.

\bibitem{Aldo_Band_05}
D.~J. Aldous and A.~Bandyopadhyay.
\newblock A survey of max-type recursive distributional equation.
\newblock {\em The Annals of Applied Probability}, 15(2):1047--1110, 2005.

\bibitem{Alsmeyer_2011}
G.~Alsmeyer.
\newblock Random recursive equations and their distributional fixed points.
\newblock Unpublished manuscript, available at
  \url{https://www.uni-muenster.de/Stochastik/lehre/SS11/StochRekGleichungen/book.pdf},
  2012.

\bibitem{Als_Big_Mei_10}
G.~Alsmeyer, J.~D. Biggins, and M.~Meiners.
\newblock The functional equation of the smoothing transform.
\newblock {\em The Annals of Probability}, 40(5):2069--2105, 2012.

\bibitem{Alsm_Mein_10a}
G.~Alsmeyer and M.~Meiners.
\newblock Fixed points of inhomogeneous smoothing transforms.
\newblock {\em Journal of Difference Equations and Applications},
  18(8):1287--1304, 2012.

\bibitem{Alsm_Mein_10b}
G.~Alsmeyer and M.~Meiners.
\newblock Fixed points of the smoothing transform: Two-sided solutions.
\newblock {\em Probability Theory and Related Fields}, 155(1-2):165--199, 2013.

\bibitem{Austin}
T.~L. Austin, R.~E. Fagen, W.~F. Penney, and J.~Riordan.
\newblock The number of components in random linear graphs.
\newblock {\em The Annals of Mathematical Statistics}, 30:747--754, 1959.

\bibitem{Biggins_98}
J~.D. Biggins.
\newblock Lindley-type equations in the branching random walk.
\newblock {\em Stochastic Processes and their Applications}, 75(1):105--133,
  1998.

\bibitem{bollobas}
B.~Bollob\'as.
\newblock A probabilistic proof of an asymptotic formula for the number of
  labelled regular graphs.
\newblock {\em European Journal of Combinatorics}, pages 311--316, 1980.

\bibitem{Bollobas2}
B.~Bollob\'as.
\newblock {\em Random graphs}.
\newblock Cambridge University Press, 2001.

\bibitem{Boll_Jan_Rio_07}
B.~Bollob\'as, S.~Janson, and O.~Riordan.
\newblock The phase transition in inhomogeneous random graphs.
\newblock {\em Random Structures \& Algorithms}, 31:3--122, 2007.

\bibitem{Brin_Page_98}
S.~Brin and L.~Page.
\newblock The anatomy of a large-scale hypertextual {W}eb search engine.
\newblock {\em Computer Networks and ISDN Systems}, 30(1-7):107--117, 1998.

\bibitem{Brittonetal}
T.~Britton, M.~Deijfen, and A.~Martin-L\"af.
\newblock Generating simple random graphs with prescribed degree distribution.
\newblock {\em Journal of Statistical Physics}, 124(6):1377--1397, 2006.

\bibitem{castellano2006zero}
C.~Castellano and R.~Pastor-Satorras.
\newblock Zero temperature glauber dynamics on complex networks.
\newblock {\em Journal of Statistical Mechanics: Theory and Experiment},
  2006(05):P05001, 2006.

\bibitem{Chen_Olv_13}
N.~Chen and M.~Olvera-Cravioto.
\newblock Directed random graphs with given degree distributions.
\newblock {\em Stochastic Systems}, 3:147--186, 2013.

\bibitem{chen2007finding}
P.~Chen, H.~Xie, S.~Maslov, and S.~Redner.
\newblock Finding scientific gems with google's pagerank algorithm.
\newblock {\em Journal of Informetrics}, 1(1):8--15, 2007.

\bibitem{Chunglu}
F.~Chung and L.~Lu.
\newblock Connected components in random graphs with given expected degree
  sequences.
\newblock {\em Annals of Combinatorics}, 6:125--145, 2002.

\bibitem{Chunglu2}
F.~Chung and L.~Lu.
\newblock The average distances in random graphs with given expected degrees.
\newblock In {\em Proceedings of National Academy of Sciences}, volume~99,
  pages 15879--15882, 2002a.

\bibitem{Chunglu3}
F.~Chung and L.~Lu.
\newblock The volume of the giant component of a random graph with given
  expected degrees.
\newblock {\em SIAM Journal on Discrete Mathematics}, 20(2):395--411, 2006a.

\bibitem{Chunglu4}
F.~Chung and L.~Lu.
\newblock {\em Complex graphs and networks}, volume 107.
\newblock CBMS Regional Conference Series in Mathematics, 2006b.

\bibitem{damonte2019systemic}
L.~Damonte, G.~Como, and F.~Fagnani.
\newblock Systemic risk and network intervention.
\newblock {\em arXiv preprint arXiv:1912.08631}, 2019.

\bibitem{de1992isotropic}
M.~J. de~Oliveira.
\newblock Isotropic majority-vote model on a square lattice.
\newblock {\em Journal of Statistical Physics}, 66(1-2):273--281, 1992.

\bibitem{degroot1974reaching}
M.~H. DeGroot.
\newblock Reaching a consensus.
\newblock {\em Journal of the American Statistical Association},
  69(345):118--121, 1974.

\bibitem{deo2019limiting}
A.~Deo and S.~Juneja.
\newblock Limiting distributional fixed points in systemic risk graph models.
\newblock In {\em 2019 Winter Simulation Conference}, pages 878--889. IEEE,
  2019.

\bibitem{Devroye_01}
L.~Devroye.
\newblock On the probabilistic worst-case time of {FIND}.
\newblock {\em Algorithmica}, 31:291--303, 2001.

\bibitem{Durrett1}
R.~Durrett.
\newblock {\em Random graph dynamics, Cambridge Series in Statistics and
  Probabilistic Mathematics}.
\newblock Cambridge University Press, 2007.

\bibitem{Erdos}
P.~Erd\H{o}s and A.~R\'enyi.
\newblock On random graphs.
\newblock {\em Publicationes Mathematicae (Debrecen)}, 6:290--297, 1959.

\bibitem{Fill_Jan_01}
J.A. Fill and S.~Janson.
\newblock Approximating the limiting {Q}uicksort distribution.
\newblock {\em Random Structures Algorithms}, 19(3-4):376--406, 2001.

\bibitem{frasca2013gossips}
P.~Frasca, C.~Ravazzi, R.~Tempo, and H.~Ishii.
\newblock Gossips and prejudices: Ergodic randomized dynamics in social
  networks.
\newblock {\em IFAC Proceedings Volumes}, 46(27):212--219, 2013.

\bibitem{friedkin1990social}
N.~E. Friedkin and E.~C. Johnsen.
\newblock Social influence and opinions.
\newblock {\em Journal of Mathematical Sociology}, 15(3-4):193--206, 1990.

\bibitem{friedkin1999influence}
N.~E. Friedkin and E.~C. Johnson.
\newblock Social influence networks and opinion change.
\newblock {\em Advances in Group Processes}, 16(1):1--29, 1999.

\bibitem{Grav_etal_20}
A.~Garavaglia, R.~van~der Hofstad, and N.~Litvak.
\newblock Local weak convergence for pagerank.
\newblock {\em The Annals of Applied Probability}, 30(1):40--79, 2020.

\bibitem{Gilbert}
E.~N. Gilbert.
\newblock Random graphs.
\newblock {\em The Annals of Mathematical Statistics}, 30:1141--1144, 1959.

\bibitem{gleich2015pagerank}
D.~F. Gleich.
\newblock Pagerank beyond the web.
\newblock {\em SIAM Review}, 57(3):321--363, 2015.

\bibitem{gyongyi2004combating}
Z.~Gyongyi, H.~Garcia-Molina, and J.~Pedersen.
\newblock Combating web spam with trustrank.
\newblock In {\em Proceedings of the 30th international conference on very
  large data bases (VLDB)}, 2004.

\bibitem{Janson}
S.~Janson, T.~Luczak, and A.~Rucinski.
\newblock {\em Random graphs}.
\newblock Wiley-Interscience, 2000.

\bibitem{Jel_Olv_10}
P.~R. Jelenkovi\'c and M.~Olvera-Cravioto.
\newblock Information ranking and power laws on trees.
\newblock {\em Advances in Applied Probability}, 42(4):1057--1093, 2010.

\bibitem{jing2008visualrank}
Y.~Jing and S.~Baluja.
\newblock Visualrank: Applying pagerank to large-scale image search.
\newblock {\em IEEE Transactions on Pattern Analysis and Machine Intelligence},
  30(11):1877--1890, 2008.

\bibitem{Kar_Kel_Suh_94}
F.~I. Karpelevich, M.~Y. Kelbert, and Y.~M. Suhov.
\newblock Higher-order {L}indley equations.
\newblock {\em Stochastic Process. Appl.}, 53:65--96, 1994.

\bibitem{lacker2020local}
D.~Lacker, K.~Ramanan, and R.~Wu.
\newblock Local weak convergence and propagation of ergodicity for sparse
  networks of interacting processes, 2020.

\bibitem{Lee_Olv_20}
J.~Lee and M.~Olvera-Cravioto.
\newblock Page{R}ank on inhomogeneous random digraphs.
\newblock {\em Stochastic Processes and their Applications}, 130(4):1--57,
  2020.

\bibitem{Mac_Stu_Swa_18}
T.~Mach, A.~Sturm, and J.~M. Swart.
\newblock A new characterization of endogeny.
\newblock {\em Mathematical Physics, Analysis and Geometry}, 21(4):30, 2018.

\bibitem{Norros}
I.~Norros and H.~Reittu.
\newblock On a conditionally {P}oissonian graph process.
\newblock {\em Advances in Applied Probability}, 38(1):59--75, 2006.

\bibitem{Olvera_12}
M.~Olvera-Cravioto.
\newblock Tail behavior of solutions of linear recursions on trees.
\newblock {\em Stochastic Processes and Applications}, 122(4):1777--1807, 2012.

\bibitem{Olvera_20}
M.~Olvera-Cravioto.
\newblock Pagerank’s behavior under degree correlations.
\newblock {\em The Annals of Applied Probability}, 31(3):1403--1442, 2021.

\bibitem{Olvera_21}
M.~Olvera-Cravioto.
\newblock Strong couplings for static locally tree-like random graphs.
\newblock {\em arXiv:2102.10673}, pages 1--27, 2021.

\bibitem{Olv_Rui_20}
M.~Olvera-Cravioto and O.~Ruiz-Lacedelli.
\newblock Stationary waiting time in parallel queues with synchronization.
\newblock {\em Mathematics of Operations Research}, 2020.

\bibitem{ravazzi2014ergodic}
C.~Ravazzi, P.~Frasca, R.~Tempo, and H.~Ishii.
\newblock Ergodic randomized algorithms and dynamics over networks.
\newblock {\em IEEE Transactions on Control of Network Systems}, 2(1):78--87,
  2014.

\bibitem{Esker_Hofs_Hoog}
H.~van~den Esker, R.~van~der Hofstad, and G.~Hooghiemstra.
\newblock Universality for the distance in finite variance random graphs.
\newblock {\em Journal of Statistical Physics}, 133:169--202, 2008.

\bibitem{Hofstad1}
R.~van~der Hofstad.
\newblock {\em Random graphs and complex networks}.
\newblock Cambridge University Press, 2016.

\bibitem{vilela2017majority}
A.~L.~M. Vilela and A.~J.~F. de~Souza.
\newblock Majority-vote model with a bimodal distribution of noises in
  small-world networks.
\newblock {\em Physica A: Statistical Mechanics and its Applications},
  488:216--223, 2017.

\bibitem{vilela2019majority}
A.~L.~M. Vilela, C.~Wang, K.~P. Nelson, and H.~E. Stanley.
\newblock Majority-vote model for financial markets.
\newblock {\em Physica A: Statistical Mechanics and its Applications},
  515:762--770, 2019.

\bibitem{villani2008optimal}
C.~Villani.
\newblock {\em Optimal transport: old and new}, volume 338.
\newblock Springer Science \& Business Media, 2008.

\bibitem{Volk_Litv_10}
Y.~Volkovich and N.~Litvak.
\newblock Asymptotic analysis for personalized web search.
\newblock {\em Advances in Applied Probability}, 42(2):577--604, 2010.

\bibitem{yang2017innovation}
M.~Yang, X.~Qu, Z.~Cao, and X.~Yang.
\newblock Innovation governs everything eventually: Extensions of the degroot
  model.
\newblock {\em Acta Mathematicae Applicatae Sinica, English Series},
  33(1):35--42, 2017.

\end{thebibliography}
\end{document}